\documentclass{article}
\usepackage{lineno}
\usepackage{amsmath}
\usepackage{amsfonts}
\usepackage{amssymb}
\usepackage{color}
\usepackage{todonotes}
\usepackage{epstopdf}
\usepackage{subfigure}
\usepackage{dsfont,bm}
\usepackage{makeidx}
\usepackage{graphicx}
\usepackage{lmodern}
\usepackage{amsthm}
\usepackage{color}
\usepackage{epstopdf}
\usepackage{subfigure}

\usepackage{enumitem}
\usepackage[numbers,sort&compress]{natbib}
\bibpunct[, ]{[}{]}{,}{n}{,}{,}
\makeatletter
\def\NAT@def@citea{\def\@citea{\NAT@separator}}
\makeatother

\newcommand{\abs}[1]{\left\vert#1\right\vert}

\newtheorem{theorem}{Theorem}[section]

\newtheorem{lemma}[theorem]{Lemma}

\theoremstyle{definition}

\theoremstyle{remark}

\begin{document}
\noindent{\bf  Statistical inference for fractional diffusion process with random effects at discrete observations}\\ \ \\ 
\small{{\bf El Omari Mohamed \textsuperscript{\textbf{a}}, Hamid El Maroufy\textsuperscript{a},
Christiane Fuchs\textsuperscript{\textbf{bcd}}}
\\ \ \\ 
\textsuperscript{a}Department of  Applied Mathematics, Faculty of Sciences and Technics, Sultan Mouly Slimane University, Morocco; \textsuperscript{\textbf{b}}Bielefeld University, Faculty of Business Administration and Economics, Bielefield, Germany; \textsuperscript{\textbf{c}}Helmholtz Zentrum Munchen, German Research Center for Environmental Health GmbH, Institute of Computational Biology, Neuherberg, Germany; \textsuperscript{\textbf{d}}Chair of Mathematical Modeling of Biology Systems, Technisch Universit{\"a}t M{\"u}nchen,   Garching, Germany.}
\begin{abstract}
This paper deals with the problem of inference associated with linear fractional diffusion process with random effects in the drift. In particular we are concerned  with the maximum likelihood estimators (MLE)  of the random effect parameters. First of all, we estimate the Hurst parameter $H\in (0,1)$ from one single subject. Second, assuming the Hurst index $H\in (0,1)$ is known, we derive the MLE and examine their asymptotic behavior as the number of subjects under study becomes large, with random effects normally distributed.
\paragraph*{Keywords}Asymptotic normality, Fractional Brownian motion, Long-range memory process,  Random effect model, Strong consistency.
 \paragraph*{Mathematics Subject Classification} 62F12 · 60G22
\end{abstract}
\section{Introduction}\label{Sec.1}
Parameteric and nonparameteric estimation in the context of random effects models has been recently investigated by many authors (e.g. \cite{Antic et  al 2009,Comte and Samson 2012,Delattre et al 2012,Ditlevsen and De Gaetano 2005a,Nie 2007,Picchini and Ditlevsen 2011}). In these models, the noise is represented by a Brownian motion characterized  by independence property of its increments. Such a property is not valid for long-memory phenomena arising in a variety of different scientific fields, including hydrology \cite{McLeod and hipel 1978}, biology \cite{Collins and De Luca 1994}, medicine \cite{Kuklinski et al 1989}, economics \cite{Granger 1966} or traffic network \cite{Willinger et al 1995}. As a result self-similar processes have been used to successfully model data exhibiting long-range dependence. Among the simplest models that display long-range dependence, one can consider the fractional Brownian motion (fBm), introduced in the statistics community by Mandelbrot and Van Ness \cite{Mandelbort et al 1968}. A normalized fBm with the Hurst index $H\in (0,1)$ is centered Gaussian process $\displaystyle \left( W^H_t~:~t\geq 0 \right)$ having the covariance
$$\mathbb{E}\left( W^H_s W^H_t\right)=\frac{1}{2}\left(t^{2H}+s^{2H}-\abs{t-s}^{2H} \right).$$
In modeling, the problems of statistical estimation of model parameters are of particular importance, so the growing number of papers devoted to statistical methods for equations with fractional noise is not surprising. We will cite only few of them; further references can be found in \cite{Mishura 2008,Prakasa 2010}. In \cite{Kleptsyna and Le Breton 2002} the authors proposed and studied maximum likelihood estimators for fractional Ornstein-Uhlenbeck process. Related results were obtained by Prakasa \cite{Prakasa  2003}, where a more general model was considered. In \cite{Hu and Nualart 2010} the author proposed a least squares (LS) estimator for fractional Ornstein-Uhlenbeck process and proved its asymptotic normality. Recently, the same results are obtained using the same approach (LS) in \cite{Xiao and Yu 2018}, for the fractional Vasicek model with long-memory. \bigskip
It is worth to mentioning the papers \cite{Hu et al 2011,Tudor and Viens 2007} that deal with the whole range of Hurst parameter $H\in (0,1)$. Meanwhile,  we have cited other   papers that only the case where  $H >1/2$ (which corresponds to long range dependence); recall that in the case $H =1/2$, we get a classical diffusion extensively treated in literature \cite{Liptser and Shiryaev 2001}.\\ \ \\
This paper deals with statistical estimation of population parameters for fractional SDE's with random effects. To our knowledge, this problem has not been yet investigated. Precisely, we concider only a fractional diffusion processes of the form
\begin{equation}\label{Type}
\displaystyle X_t=x+ \int_0^t \left( a(X_s)+\phi b(X_s) \right) ds+W^H_t,
\end{equation}
where $\phi$ is a random variable relying on parameter $\theta$ to be estimated, and $W^H$ is a normalized fBm with Hurst parameter $H$ to be estimated. We study the additive linear case,  $b(x)\equiv 1$, when $\displaystyle \phi\sim\mathcal{N}(\mu,\sigma^2)$. The estimators $\widehat{\mu}$, $\widehat{\sigma^2}$ and  $\widehat{H}$ respectively of $\mu$, $\sigma^2$ and  $H$ respectively  are constructed and their asymptotic behaviors are investigated. The model  (\ref{Type}) is simple  and we can derive explicit estimators, also the model  generalizes the model considered  in \cite{Hu et al 2011}, while the techniques used here to investigate asymptotic properties are elementary  (due to the incorporation of the  random effects, hence we avoid the Malliavin techniques),  which gives as  a first reason to  choose it. The second reason is that (\ref{Type}) is widely applied in various fields. in fact the Vasiceck model is an example of type (\ref{Type}). The third reason is that the estimation of the population parameters requires few observations per subject, which coincide with  several natural phenomena  where the repeated measurements are rarely available if not impossible. Finally let's note that nonparametric estimation has been realized recently by us for similar model\cite{El Omari et al 2019}.\bigskip
The rest of the  paper  is organized as follows. In Section \ref{Sec.2}, we introduce the model and some preliminaries about the likelihood function. In  Section \ref{Sec.3} we derive the parameters estimators and we establish consistency and asymptotic normality. The simulations are presented in Section  \ref{Sec.4}  while  Section \ref{Sec.5} contains some concluding remarks and gives directions of further research.\\
Throughout the paper the notations  $\Longrightarrow$, $\stackrel{\mathbb{P}-as}{\Longrightarrow}{}$ and $\stackrel{\mathcal{D}}{\Longrightarrow}{}$ mean, respectively, simple convergence,  convergence almost surely with respect to the probability measure $\mathbb{P}$ and  convergence in distribution.
\section{Model and Preliminary results}\label{Sec.2}
Before introducing our estimation techniques, we first state some basic facts about
fractional Brownian motions and likelihood function.
Let $\displaystyle\left(\Omega,\mathcal{F},(\mathcal{F}_t^i),\mathbb{P}\right)$ be a stochastic basis satisfying the usual conditions. The natural filtration of a stochastic process is understood as the $\mathbb{P}$-completion of the filtration generated by this process. Let $\displaystyle W^{H,i}=\left( W^{H,i}(t)~,~t\leq T :\right)$, $i=1,\cdots,N$ be $N$ independent normalized fractional Brownian motion (fBm) with a common Hurst parameter $\displaystyle H\in (0,1)$. Let $\displaystyle \phi_1,\cdots,\phi_N$ be $N$ independent and identically distributed (i.i.d) $\mathds{R}$-valued random variables on the common probability space $\displaystyle\left(\Omega,\mathcal{F},\mathbb{P} \right)$ independent of  $\displaystyle \left( W^{H,1},\cdots,W^{H,N}\right)$. Consider $N$ subjects $\displaystyle \left(X^i(t),\mathcal{F}_t^i,t\leq T \right)$ with dynamics ruled by the following general linear stochastic differential equations:
\begin{eqnarray}\label{Eq1}
\displaystyle dX^i(t) &=& \left(a(X^i(t))+\phi_i b(X^i(t)) \right)dt+dW^{H,i}(t)\\
\displaystyle \nonumber  X^i(0) &=& x^i\in \mathds{R},~~i=1,\cdots,N,
\end{eqnarray}
where $a(\cdot)$ and $b(\cdot)$ are supposed to be known in their own spaces. Let the random effects $\phi_i$ be $ \mathcal{F}_0^i$-measurable with common density $g(\varphi,\theta)d\nu( \varphi)$, where $\nu$ is some dominating measure on $\mathds{R}$ and $\theta$ is unknown parameter. Set $\theta\in U$, where $U$ is an open set in $\mathds{R}^d$. Sufficient conditions for the existence and uniqueness of solutions to (\ref{Eq1}) can be found in \cite[p. 197]{Mishura 2008} and references therein.\bigskip
Let $C_T$ denote the space of real continuous functions $\displaystyle \left( x(t)~:~t\in [0,T] \right)$ defined on $[0,T]$ endowed with $\sigma$-field $\mathcal{B}_T$. The $\sigma$-field $\mathcal{B}_T$ is associated with the topology of uniform convergence on $[0,T]$. We introduce the distribution $\displaystyle \mu_{X^i_{\varphi,H}}$ on $\displaystyle \left(C_T,\mathcal{B}_T \right)$ of the process $\left( X^i|\phi_i=\varphi\right)$. On $\mathds{R}\times C_T$, $\displaystyle Q^i_{\theta,H}=g(\varphi,\theta)d\nu\otimes \mu_{X^i_{\varphi,H}}$ denotes the joint distribution of $(\phi_i,X^i)$. Let $\mathbb{P}^i_{\theta,H}$ be the marginal distribution of $\left(X^i(t)~:~t\leq T \right)$ on $\left(C_T,\mathcal{B}_T \right)$. Since the subjects are independent (this is inherited from the independence of $\phi_i$ and $ W^{H,i}$), the distribution of the whole sample $\left(X^i(t)~:~t\leq T,~i=1,\cdots,N \right)$ on $C_T^{\otimes N}$ is defined by $\displaystyle\mathbb{P}_{\theta,H}=\otimes_{i=1}^{N}\mathbb{P}^i_{\theta,H}$. Thus the likelihood can be defined as
\begin{equation*}
\displaystyle \Lambda(\theta,H) = \frac{\mbox{d} \mathbb{P}_{\theta,H}}{\mbox{d} \mathbb{P}} = \prod_{i=1}^{N} \frac{\mbox{d} \mathbb{P}^i_{\theta,H}}{\mbox{d}\mathbb{P}^i} ,
\end{equation*}where $\mathbb{P}=\otimes_{i=1}^{N}\mathbb{P}^i$ and $\mathbb{P}^i=\mu_{X^i_{\varphi_0,H}}$, provided that $\displaystyle \mu_{X^i_{\varphi,H}}\ll\mu_{X^i_{\varphi_0,H}}$ for some fixed $\varphi_0\in\mathds{R}$. It is well known that $\displaystyle \mu_{X^i_{\varphi,H}}$ coincides with the distribution of the process $X^{i,\varphi}$ defined by:
\begin{equation*}
\displaystyle  dX^{i,\varphi}(t) = \left(a(X^i(t))+\varphi b(X^{i,\varphi}(t)) \right)dt+dW^{H,i}(t),~ X^{i,\varphi}(0)=x^i,
\end{equation*}when $H=1/2$, since in this case the process $(X^i,\phi_i)$ is markovian (e.g. \cite{Catalot and Laredo 2013}); hence, the Girsanov formula can be applied to get the derivative $\displaystyle \frac{d\mu_{X^i_{\varphi,H}}}{d\mu_{X^i_{\varphi_0,H}}}$. When $H\neq 1/2$, the non Markovian property of the coupled process $(X^i,\phi_i)$ makes the construction of the likelihood very difficult. But in our case, the process $X^i$ is transformed into a $Y^i$ for which the law of $\displaystyle \left( Y^i | \phi_i=\varphi\right)$ coincides with the distribution of a parametric fractional diffusion process $Y^{i,\varphi}$.
\section{Construction of estimators and their asymptotic properties}\label{Sec.3}
Consider the following process
\begin{eqnarray}\label{PSY}
 \displaystyle Y^i(t) &:=& X^i(t)-x^i-\int_0^t a(X^i(s))ds ,~~~t\geq 0\\
\displaystyle &=& t\phi_i+W^{H,i}(t)\sim  \mathcal{N}\left(t\mu,t^2\sigma^2+t^{2H} \right),~~~t\geq 0\label{PSYbis}.
\end{eqnarray}
Since $\phi_i$ and $W^{H,i}$ are independent  the process $\displaystyle \left( Y^i(t)~:~t\in [0,T]\right)$ is a Gaussian process. Furthermore, for each $\varphi\in\mathds{R}$, we have $\mathbb{E}\left(Y^i(t) |\phi_i=\varphi\right)=t\varphi$ and $\mathbb{C}ov\left(Y^i(t),Y^i(s) |\phi_i=\varphi\right)=\frac{1}{2}(t^{2H}+s^{2H}-\abs{t-s}^{2H})$. For each subject $Y^i$, we consider $n$ observations $Y^i:=\left(Y^i(t_1),\cdots,Y^i(t_n) \right)'$ where $0=t_0<t_1<\cdots<t_n=T$ is a subdivision of $[0,T]$. The density of $Y^i$ given $\phi_i=\varphi$ is expressed as
\begin{equation*}
\displaystyle \Pi ( Y^i|\phi_i=\varphi,H ) = \frac{1}{\sqrt{(2\pi)^n \mbox{det}V(H)}}\exp\left(-\frac{1}{2}(Y^i-\varphi u)'V^{-1}(H)(Y^i-\varphi u)\right),
\end{equation*}where $u=(t_1,\cdots,t_n)'$ and $(V(H))_{k,l}=\mathbb{C}ov\left(Y^i(t_k),Y^i(t_l) |\phi_i=\varphi\right)$ is the common covariance matrix of the subjects $Y^i$, $i=1,\cdots,N$. The log-likelihood of the whole sample $\left( Y^1,\cdots,Y^N\right)$ is defined as
\begin{equation}\label{Eq2}
\displaystyle l(\theta,H)=\sum_{i=1}^{N} \log\int\Pi ( Y^i|\phi_i=\varphi,H )  g(\varphi,\theta) d\nu(\varphi).
\end{equation}
For a specific distribution (say $ g(\varphi,\theta) d\nu(\varphi)=\mathcal{N}(\mu,\sigma^2)$, we can solve the integrals given in (\ref{Eq2}). Indeed,
\begingroup\small
\begin{eqnarray}\label{Eq2.1}
\nonumber \int\Pi ( Y^i|\phi_i=\varphi,H )  g(\varphi,\theta) d\nu(\varphi)\hspace*{-0.3cm} & = & \hspace*{-0.3cm}
(2\pi)^{\frac{-n}{2}}\sigma^{-1} \mbox{det}(V(H))^{\frac{-1}{2}}
\left( {u}' V^{-1}(H) u+1/\sigma^2\right)^{\frac{-1}{2}} \\
  &~&  \hspace*{-3cm}\times\exp\left[-\frac{1}{2}\left( \mu^2/\sigma^2 + {Y^i}' V^{-1}(H) Y^i -\frac{({u}' V^{-1}(H)  Y^i+\mu/\sigma^2)^2}{{u}' V^{-1}(H) u+1/\sigma^2} \right)\right].
\end{eqnarray}
\endgroup
\subsection{Estimation of the Hurst parameter $H$}
Using data induced by one single subject (without loss of generality, say $Y^1$ with $t_j=\frac{j}{n},j=1,\cdots,n,~T=1$), we may construct a class of estimators of the Hurst index $H$. More precisely, for all $k>0$ and for any filter $\gamma=(\gamma_0,\cdots\gamma_l)$ of order $p\geq 2$, that is,
\begin{equation}\label{Filter}
\displaystyle \mbox{for all indices}~ 0\leq r<p;~\sum_{j=0}^{l} j^r \gamma_j=0~\mbox{ and }~\sum_{j=0}^{l} j^p \gamma_j\neq 0.
\end{equation}
Let consider the following arguments :
$ \widehat{H}(n,p,k,\gamma,Y^1) = g^{-1}_{n,k,\gamma}\left( S_n(k,\gamma)\right)$,  where $\displaystyle  S_n(k,\gamma) = \frac{1}{n-l}\sum_{i=l}^{n-1}\abs{\sum_{q=0}^{l} \gamma_q Y^1\left(\frac{i-q}{n}\right)}^k$,  $g_{n,k,\gamma}(t)=\frac{1}{n^{tk}}\lbrace\pi_t^\gamma(0)\rbrace^{k/2} E_k$ and   $
 \displaystyle \pi_t^\gamma(j) = -\frac{1}{2}\sum_{q,r}^{l}\gamma_q \gamma_r \abs{q-r+j}^{2t},
~\mbox{with}~ E_k = 2^{k/2}\Gamma(k+1/2)/ \Gamma(1/2)
$,
with  $\Gamma(x)$ is the usual gamma function. For invertibility of the function $g_{n,k,\gamma}(\cdot)$, we refer to \cite[p. 7]{Coeurjolly 2001}.
\begin{theorem}
 The following statements holds true, as the number of observations $n\longrightarrow\infty$,
\begin{description}
\item[\bf{(i)}] $\displaystyle \widehat{H}(n,p,k,\gamma,Y^1)\stackrel{\mathbb{P}-as}{\Longrightarrow}{}H$
\item[\bf{(ii)}] $\displaystyle n^{-1/2}\log(n)\left( \widehat{H}(n,p,k,\gamma,Y^1)-H\right)\stackrel{\mathcal{D}}{\Longrightarrow}{} \mathcal{N}\left(0,\frac{A(H,k,\gamma)}{k^2} \right)$, where
\begin{eqnarray*}
\displaystyle A(t,k,\gamma) &=& \sum_{j\geq 1} (c^k_{2j})^2(2j)!\sum_{i\in\mathbb{Z}}{\rho^\gamma_t(i)}^{2j}, ~\mbox{with}\\
\displaystyle c^k_{2j} &=&\frac{1}{(2j)!}\prod_{q=0}^{j-1}(k-2q),~\mbox{and}~\rho^\gamma_t(i)=\frac{\pi_t^\gamma(i)}{\pi_t^\gamma(0)}.
\end{eqnarray*}
\end{description}
\end{theorem}
\begin{proof} Following Coeurjolly \cite{Coeurjolly 2001}, we set $\displaystyle V^\gamma(i/n)=\sum_{q=0}^{l} \gamma_q W^{H,1}\left( \frac{i-q}{n}\right)$, for $i=l,\cdots,n-1$. From (\ref{Filter}) we see that  the  filter $\gamma$ is of order $p\geq 2$, $\displaystyle \sum_{q=0}^{l}  \frac{i-q}{n} \gamma_q =0$. Therefore, substituting $\displaystyle Y^1\left(\frac{i-q}{n} \right)$ by $\displaystyle\frac{i-q}{n}\phi_1+W^{H,1}\left(\frac{i-q}{n}\right)$, we obtain
\begin{eqnarray*}
\displaystyle S_n(k,\gamma) &=& \frac{1}{n-l}\sum_{i=l}^{n-1}\abs{\sum_{q=0}^{l} \gamma_q Y^1\left(\frac{i-q}{n}\right)}^k\\
\displaystyle &=& \frac{1}{n-l}\sum_{i=l}^{n-1}\abs{\sum_{q=0}^{l} \gamma_q \frac{i-q}{n}\phi_1+\sum_{q=0}^{l} \gamma_q W^{H,1}\left(\frac{i-q}{n}\right)}^k\\
\displaystyle &=& \frac{1}{n-l}\sum_{i=l}^{n-1}\abs{ V^\gamma(i/n)}^k.
\end{eqnarray*}
Hence, our estimators coincide with estimators $\widehat{H}$ based on $k$-variations of the fBm (see \cite[Proposition 2]{Coeurjolly 2001}) and the proof is complete.
\end{proof}
\subsection{Estimation of the population parameter $\theta=(\mu,\sigma^2)$}
Now, assume that $H$ is known. From the log-likelihood given by (\ref{Eq2}) and (\ref{Eq2.1}), we derive an estimator $\widehat{\mu}$ given  by
\begin{equation}
\displaystyle \widehat{\mu} = \frac{\frac{1}{N}\displaystyle\sum_{i=1}^{N} {u}' V^{-1}(H) Y^i}{u' V^{-1}(H) u}.
\end{equation}For the parameter $\sigma^2$ it sounds very difficult to derive an estimator. However, we can construct an alternative estimator and study its asymptotic behavior. Observing that $\widehat{\mu}$ is a sample mean drawn from a sequence of i.i.d random variables, one might think that sample variance could also be used to estimate $\sigma^2$. Unfortunately, simple computations shows that such a sample variance is not consistent. Thus, as an alternative  we propose the following estimator  for $\sigma^2$:
\begingroup\small
\begin{equation}
\displaystyle \widehat{\sigma^2} = \frac{1}{N} \sum_{i=1}^{N} \left( \frac{{u}' V^{-1}(H) Y^i}{u' V^{-1}(H) u}\right)^2-\frac{1}{N^2}\left(\sum_{i=1}^{N}\frac{{u}' V^{-1}(H) Y^i}{u' V^{-1}(H) u} \right)^2-\left(u' V^{-1}(H) u\right)^{-1}.
\end{equation}
\endgroup
\begin{theorem}\,\ \\
The estimator $\widehat{\mu}$ is unbaised, $\widehat{\mu} \stackrel{\mathbb{P}-as}{\Longrightarrow}{}\mu$ and  $\mathbb{V}ar(\widehat{\mu})\longrightarrow 0$ as $N\rightarrow\infty$.
\end{theorem}
\begin{proof}
Set $\epsilon^i=\left(W^{H,i}(t_1),\cdots,W^{H,i}(t_n)\right)'$. Substituting $Y^i$ by $\phi_i u+\epsilon^i$, we have
$ \displaystyle \widehat{\mu}=\frac{1}{N}\sum_{i=1}^{N}\phi_i +\frac{\frac{1}{N}\displaystyle\sum_{i=1}^{N} {u}' V^{-1}(H) \epsilon^i}{u' V^{-1}(H) u},$
so $\displaystyle \mathbb{E}(\widehat{\mu})=\mu$. For the second statement, we consider the random variables $\xi_i(H)$ defined by
\begin{equation}\label{UQU}
\displaystyle \xi_i(H)=\frac{u' V^{-1}(H) Y^i}{u' V^{-1}(H) u}.
\end{equation}
Clearly, $\xi_i(H)$ are i.i.d random variables with $\mathbb{E}(\xi_i(H))=\mu <\infty$, then by  strong law of large numbers,  $\widehat{\mu}$ converges almost surely to $\mu$ as $N\rightarrow \infty$. Set $z(H):=(z_1(H),\cdots,z_n(H))=u ' V^{-1}(H)$, we have
\begingroup\small
\begin{eqnarray*}
\displaystyle \mathbb{V}ar(\widehat{\mu})\hspace*{-0.2cm}  &=& \hspace*{-0.15cm} \mathbb{V}ar\left(\frac{1}{N}\sum_{i=1}^{N}\phi_i\right)+\frac{1}{N^2(z(H)\cdot u)^2}\mathbb{V}ar\left( \sum_{i=1}^{N} z(H)\cdot \epsilon^i\right) \\
\displaystyle\hspace*{-0.2cm}  &=& \frac{1}{N^2} \mathbb{V}ar\left(\sum_{i=1}^{N}\phi_i\right) \\
\hspace*{-0.2cm} & + & \hspace*{-0.17cm}\frac{1}{N^2(z(H)\cdot u)^2} \sum_{i,j}^{N} \mathbb{E}\left\lbrace\left( \sum_{k=1}^n z_k(H) W^{H,i}(t_k)\right)\left( \sum_{l=1}^n z_l(H) W^{H,j}(t_l)\right)\right\rbrace\\
\displaystyle \hspace*{-0.2cm}  &=&  \hspace*{-0.17cm}\frac{1}{N^2} \sum_{i=1}^{N}\mathbb{V}ar(\phi_i)+\frac{1}{N^2(z(H)\cdot u)^2}\sum_{i,j}^{N}\sum_{k,l}^n z_k(H)z_l(H)\mathbb{E} \left(   W^{H,i}(t_k) W^{H,j}(t_l)\right)\\
\displaystyle ~ \hspace*{-0.2cm}  &=&  \hspace*{-0.17cm}\frac{\sigma^2}{N}+\frac{1}{N^2(z(H)\cdot u)^2} \sum_{i}^{N}\sum_{k,l}^n z_k(H)z_l(H)\mathbb{E}\left(   W^{H,i}(t_k) W^{H,i}(t_l)\right)\\
\displaystyle  \hspace*{-0.2cm}  &=& \hspace*{-0.17cm}  \frac{\sigma^2}{N}+\frac{1}{N^2(z(H)\cdot u)^2} \sum_{i}^{N}\sum_{k,l}^n \frac{1}{2}z_k(H)z_l(H)\left( t_k^{2H}+t_l^{2H}-\abs{t_k-t_l}^{2H}\right)\\
\displaystyle \hspace*{-0.2cm} &=& \hspace*{-0.17cm} \frac{\sigma^2}{N}+\frac{1}{N^2(z(H)\cdot u)^2} N z(H)V(H)z(H)'= \sigma^2+\frac{Nu' V^{-1}(H)V(H)V^{-1}(H) u}{N^2(z(H)\cdot u)^2} \\
\displaystyle  \hspace*{-0.2cm} &=&  \hspace*{-0.17cm} \frac{\sigma^2}{N}+\frac{1}{N u' V^{-1}(H) u}\longrightarrow 0 ~~\mbox{as}~~N\rightarrow\infty.
\end{eqnarray*}
\endgroup
\end{proof}
Before, we establish  the bias of $\widehat{\sigma^2}$ the estimator of $\sigma^2$, we first give the following result:
\begin{lemma}
\begingroup\small
\begin{equation*}
\displaystyle \mathbb{E}(\xi_1(H))^2=\sigma^2+\mu^2+\frac{1}{u ' V^{-1}(H) u}~~\mbox{and}~~ \displaystyle \mathbb{E}\left( \sum_{i=1}^{N}\xi_i(H)\right)^2=N\sigma^2+N^2\mu^2+\frac{N}{u ' V^{-1}(H) u},
\end{equation*}\endgroup where $\xi_i(H)$ are random variables given by (\ref{UQU}).
\end{lemma}
\begin{proof}
Substituting $Y^1$ by $\phi_1 u+\epsilon^1$,  the independence of $\phi_1 $ and $\epsilon^1$ gets
\begingroup\small
\begin{eqnarray*}
\displaystyle \mathbb{E}(\xi_1(H))^2 &=& \mathbb{E}\left( \phi_1+\frac{u' V^{-1}(H)\epsilon^1}{u' V^{-1}(H) u}\right)^2\\
\displaystyle &=& \mathbb{E}\phi_1^2+\mathbb{E}\left(\frac{u' V^{-1}(H)\epsilon^1}{u' V^{-1}(H) u} \right)^2\\
\displaystyle &=& \sigma^2+\mu^2+\frac{1}{u ' V^{-1}(H) u}.
\end{eqnarray*}\endgroup
For the last equality we used the same techniques as in the proof of Theorem 2. For the second statement; by using the random variables $z_i(H)'$s  defined previously, we have
\begingroup\small
\begin{eqnarray*}
\displaystyle \mathbb{E}\left( \sum_{i=1}^{N}\xi_i(H)\right)^2 &=& \mathbb{E}\left( \sum_{i=1}^{N}\phi_i + \sum_{i=1}^{N} \frac{z(H)\cdot \epsilon^i}{z(H)\cdot u} \right)^2 \\
\displaystyle &=& \mathbb{E}\left( \sum_{i=1}^{N}\phi_i\right)^2 + \mathbb{E}\left( \sum_{i=1}^{N} \frac{z(H)\cdot \epsilon^i}{z(H)\cdot u}\right)^2\\
\displaystyle &=& \sum_{i=1}^{N}\mathbb{E}\phi_i^2+2\sum_{i<j}^{N}\mathbb{E}(\phi_i\phi_j)+\frac{1}{(u ' V^{-1}(H) u)^2}\mathbb{V}ar\left( \sum_{i=1}^{N} z(H)\cdot \epsilon^i\right)\\
\displaystyle &=& N\sigma^2+N^2\mu^2+\frac{N}{u ' V^{-1}(H) u}.
\end{eqnarray*}
\endgroup
\end{proof}
\begin{theorem}
The estimator $\widehat{\sigma^2}$ is asymptotically unbiased,  $\widehat{\sigma^2}\stackrel{\mathbb{P}-as}{\Longrightarrow}{}\sigma^2$   and
 $
\displaystyle \mathbb{V}ar(\widehat{\sigma^2}) = \frac{2(N-1)}{N^2}\left( \sigma^2 + \frac{1}{u ' V^{-1}(H) u}\right)^2\Longrightarrow 0
$ as $N\rightarrow \infty$.
\end{theorem}
\begin{proof}
By virtue of Lemma 3, we get
\begingroup\small
\begin{eqnarray*}
\displaystyle \mathbb{E}(\widehat{\sigma^2}) &=& \frac{1}{N}\sum_{i=1}^N \left( \sigma^2+\mu^2+\frac{1}{u ' V^{-1}(H) u}\right)
 -\frac{1}{N^2}\left(N\sigma^2+N^2\mu^2+\frac{N}{u ' V^{-1}(H) u}  \right)\\
 &~&~~~~~~~~~~-\frac{1}{u ' V^{-1}(H) u}\\
 \displaystyle &=&  \frac{N-1}{N}\sigma^2-\frac{1}{N(u ' V^{-1}(H) u)}\Longrightarrow \sigma^2 ~~\mbox{as}~~N\rightarrow\infty.
\end{eqnarray*}
\endgroup
Applying the strong law of large numbers and the continuous mapping theorem for almost sure convergence, we get
\begingroup\small
\begin{eqnarray*}
\displaystyle \widehat{\sigma^2}&=& \frac{1}{N}\sum_{i=1}^N  \xi_i(H)^2 - \left( \frac{1}{N}\sum_{i=1}^N  \xi_i(H)\right)^2-\frac{1}{u ' V^{-1}(H) u} \\
\displaystyle &\stackrel{\mathbb{P}-as}{\Longrightarrow}{} &\mathbb{E}( \xi_1(H))^2-\mathbb{E}^2( \xi_1(H))-\frac{1}{u ' V^{-1}(H) u}=\mathbb{V}ar(\xi_1(H))-\frac{1}{u ' V^{-1}(H) u}\\
\displaystyle &=& \mathbb{V}ar\left( \phi_1+\frac{u' V^{-1}(H)\epsilon^1}{u' V^{-1}(H) u}\right)-\frac{1}{u ' V^{-1}(H) u}\\
\displaystyle &=& \mathbb{V}ar\phi_1+\mathbb{V}ar\left(\frac{u' V^{-1}(H)\epsilon^1}{u' V^{-1}(H) u}\right)-\frac{1}{u ' V^{-1}(H) u}\\
\displaystyle &=& \sigma^2+\mathbb{E}\left(\frac{u' V^{-1}(H)\epsilon^1}{u' V^{-1}(H) u}\right)^2-\frac{1}{u ' V^{-1}(H) u}=\sigma^2.
\end{eqnarray*}\endgroup
Similar computations lead to
\begin{eqnarray*}
\displaystyle \mathbb{V}ar(\widehat{\sigma^2}) &=& \frac{N-1}{N^3}\left( (N-1)\mathbb{E}(\xi_1(H)-\mu)^4-(N-3)\beta^2 \right)\\
\displaystyle &=& \frac{2(N-1)}{N^2}\beta^2,
\end{eqnarray*}where $\displaystyle \beta=\mathbb{V}ar(\xi_1(H))=\sigma^2+\frac{1}{u' V^{-1}(H) u}$. In the last equality we used the fact that $\displaystyle (\xi_1(H)-\mu)$ is a centered Gaussian random variables with variance $\beta$.
\end{proof}
\paragraph*{Remark}
 For the case of continuous observation with horizon $T$, we propose the following estimator $\widetilde{\mu}(T,N)$ defined by
$$
\displaystyle \widetilde{\mu}(T,N) =\frac{1}{NT}\sum_{i=1}^{N} Y^i(T)
.$$
It is easy to see that
$
\displaystyle \mathbb{E}\abs{\frac{1}{T}Y^i(T)-\phi_i}^2 \leq  \frac{1}{T^{2-2 H}} \longrightarrow 0$ as $T \longrightarrow \infty$ and
$\widetilde{\mu}(T,N)$ is consistent when $T,N\rightarrow\infty$. The reason we choose this double asymptotic framework, is that we proceed in two steps; in the first step
we estimate random effects $\phi_i$ as the horizon $T$ increases to $\infty$, then we use the empirical mean and variance to estimate $\theta=(\mu,\sigma^2)$, where the random effects are replaced by their estimators.
\begin{theorem}
 The estimators $\widehat{\mu}$ and $\widehat{\sigma^2}$ are asymptotically normal, {\it i.e.}
\begin{equation}\label{Eq4}
\displaystyle \sqrt{N}\left(\widehat{\mu}-\mu\right)\stackrel{\mathcal{D}}{\Longrightarrow}{} \mathcal{N}\left(0,\sigma^2+\frac{1}{u' V^{-1}(H) u} \right)\,\, \mbox{as}  ~N\rightarrow\infty,
\end{equation} and
\begin{equation}\label{Eq5}
\displaystyle \sqrt{\frac{N}{2}}\left(\widehat{\sigma^2}-\sigma^2\right)\stackrel{\mathcal{D}}{\Longrightarrow}{}  \mathcal{N}\left(0,\left(\sigma^2+\frac{1}{u' V^{-1}(H) u} \right)^2\right) \,\, \mbox{as}  ~N\rightarrow\infty.
\end{equation}
\end{theorem}
\begin{proof}
Since $\widehat{\mu}$ is the average of $N$ i.i.d random variables with finite mean and finite variance, (\ref{Eq4}) follows imediately from the central limit theorem.  In order to show (\ref{Eq5}) we consider the following random variables $\displaystyle \widetilde{\xi}_i(H)=\sqrt{\frac{N}{N-1}}(\xi_i(H)-\widehat{\mu})$, $i=1,\cdots,N$ and set $\displaystyle \beta=\sigma^2+\frac{1}{u' V^{-1}(H) u} $. We see that  $\displaystyle \left( \widetilde{\xi}_i(H)~,~i=1,2,\cdots\right)$ is centered  Gaussian process, with variance  $\mathbb{V}ar(\widetilde{\xi}_i(H))=\mathbb{E}(\widetilde{\xi}_i(H)^2)=\beta$, and
$ \displaystyle \mbox{and}~~~~~~\mathbb{V}ar(\widetilde{\xi}_i(H)^2)=2\beta^2$. So
using the  strong law of large numbers, we have
$
\displaystyle \widetilde{\sigma^2}=\frac{1}{N}\sum_{i=1}^{N} \widetilde{\xi}_i(H)^2  \stackrel{\mathbb{P}-as}{\Longrightarrow}{} \beta
$  as $N\rightarrow\infty$.  Furthermore,  the central limit theorem  leads to
$
\displaystyle \sqrt{N}\left(\widetilde{\sigma^2}-\beta\right) \stackrel{\mathcal{D}}{\Longrightarrow}{}\mathcal{N}(0,2\beta^2)
$ as $N\rightarrow\infty$.  Since
$
\displaystyle \sqrt{\frac{N}{2}}\left(\widehat{\sigma^2}-\sigma^2\right) = \alpha_N \sqrt{N}\left(\widetilde{\sigma^2}-\beta\right)-
\varepsilon_N,~\mbox{where}~\alpha_N=\frac{N-1}{\sqrt{2}N}~\mbox{and}~\varepsilon_N=
\frac{\beta}{\sqrt{2 N}}
$. Therefore, using Slutsky theorem, the convergence in  (\ref{Eq5}) is easily concluded.
\end{proof}
\section{Simulations}\label{Sec.4}
We will implement the two population parameter estimators for the model that we have studied to show their empirical behavior. We will simulate the observed vectors $Y^i$ using (\ref{PSYbis})  for two   numbers of subjects $N=50$ and $N=500$  with different lengths of observations per subject; $n=2^2$, $n=2^5$ and $n=2^8$. The fractional Brownian motions are simulated  as in \cite{Kroese and Botev 2013}. The experiment is as follows : we set $H$  equal to  $0.15$, $0.5$ and $0.85$.  For each case, replications involving $400$ samples are obtained by  resampling $n$ trajectories of $Y^i$.\\ \ \\
The averages of the estimators and their exact against empirical standard deviations are reported in the Tables \ref{Table 1}-\ref{Table 3}.  The tables show that the parameter estimations  are generally much closer to their true values as the number of subjects increases. Figures \ref{figure 1}-\ref{figure 3} display the histograms densities of the  estimators, which reveal the convergence toward a limit distribution also as $N$ is sufficiently  large, this confirms what was established before. Looking at Table \ref{Table 1}, we see that the estimating for $\sigma^2$ is  not really close to exact values when there are very few observations ($n \leq 2^3$) per subject when $H = 0.85$, this case has been observed every time when $H$ becomes large than $1/2$. In this situation,  for   the real cases where the true value of $\sigma^2$ is not available, it will be better to choose $n$  as large as possible ($n\geq 2^4$) but this leads to huge computational cost for large values of $N$. Yet, to keep the balance between the computational cost and goodness of fit, a small values of $n$ and sufficiently large values of $N$ should be considered.
\begin{table}[h]
\resizebox{\textwidth}{!}{%
\begin{tabular}{lllll}
\hline
True  values                   \hspace{2cm}& $H=0.15$           & $H=0.50$   & $H=0.85$ \\ \cline{1-1}
\begin{tabular}[c]{@{}l@{}} $N=50$  \end{tabular} ~~~&   \begin{tabular}[c]{@{}l@{}}Mean (Std. dev.)\end{tabular} ~~~&  \begin{tabular}[c]{@{}l@{}}Mean (Std. dev.)\end{tabular} ~~~&  \begin{tabular}[c]{@{}l@{}} Mean and  Std. dev.'s \end{tabular} &  ~~~
\\  \cline{2-4}
\hline
\begin{tabular}[c]{@{}l@{}}$\mu=-2$\\ $\sigma^2=1$\end{tabular} ~~~~& \begin{tabular}[c]{@{}l@{}}
-1.9902   (\color{red}{0.1456} \color{blue}{0.1430})\\~0.9744 (\color{red}{0.2099  }    \color{blue}{0.1942})\end{tabular} ~~~~& \begin{tabular}[c]{@{}l@{}}~-1.9964     (\color{red}{0.1549 }     \color{blue}{0.1594})\\ ~ 1.0303   (\color{red}{0.2376 }  \color{blue}{0.2494})   \end{tabular} ~~~~& \begin{tabular}[c]{@{}l@{}} ~-1.9820     (\color{red}{0.1795 }    \color{blue}{0.2009})    \\~1.3314     (\color{red}{0.3191 } \color{blue}{0.3891})\end{tabular} &
\\
            $N=500$                       &                                   &  &  \\ \cline{1-1}
\begin{tabular}[c]{@{}l@{}}$\mu=-2$\\ $\sigma^2=1$\end{tabular}   ~~~~& \begin{tabular}[c]{@{}l@{}}
-2.0009     (\color{red}{ 0.0460} \color{blue}{ 0.0441})\\~    0.9964     (\color{red}{0.0670}    \color{blue}{0.0689})\end{tabular} ~~~~& \begin{tabular}[c]{@{}l@{}}  -1.9986      (\color{red}{0.0490}    \color{blue}{0.0515})\\~ 1.0442      (\color{red}{0.0758}    \color{blue}{0.0836})\end{tabular} ~~~~& \begin{tabular}[c]{@{}l@{}}~-1.9985     (\color{red}{0.0568}  \color{blue}{0.0634}) \\~    1.2022     (\color{red}{0.1018}    \color{blue}{0.1228})\end{tabular} &  ~~~~\\
\hline
\end{tabular}}
\caption{The means with exact (red) and empirical (blue) standard deviations of estimators $\widehat{\mu}$, $\widehat{\sigma^2}$ based on $400$ samples, with true values $(\mu_0,\sigma^2_0)=(-2,1)$, $\bm{(T,n)=(5,2^2)}$, and different values of $N$ ($=50;500$).(For interpretation of the references to colour in this table the reader is referred to the electronic version of
this article).} \label{Table 1}
\end{table}

\begin{figure}{}
 \begin{center}
   \includegraphics[scale=1,totalheight=4cm,width=0.35\textwidth ]{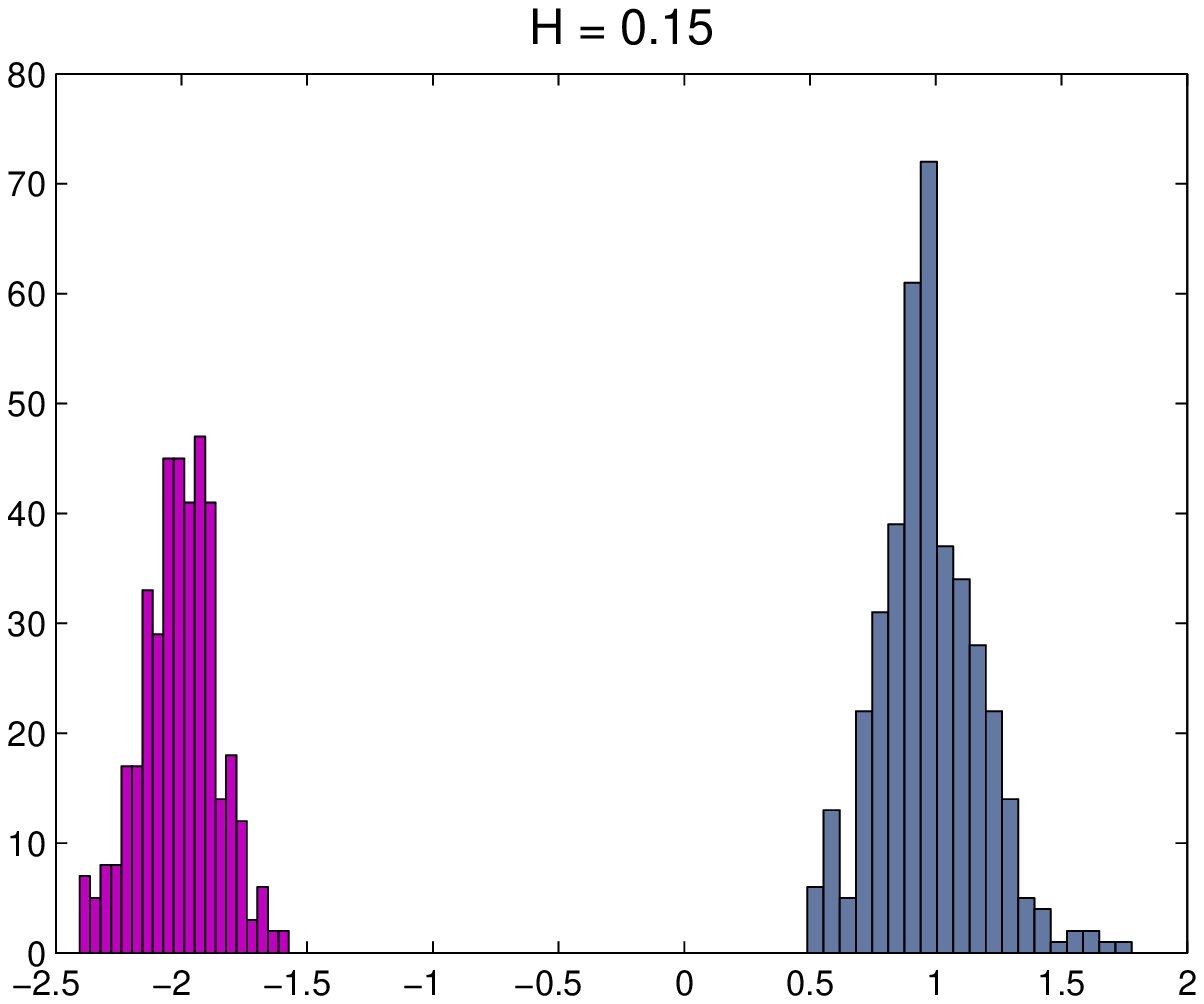}\hspace*{-0.27cm}
  \includegraphics[scale=1,totalheight=4cm,width=0.35\textwidth ]{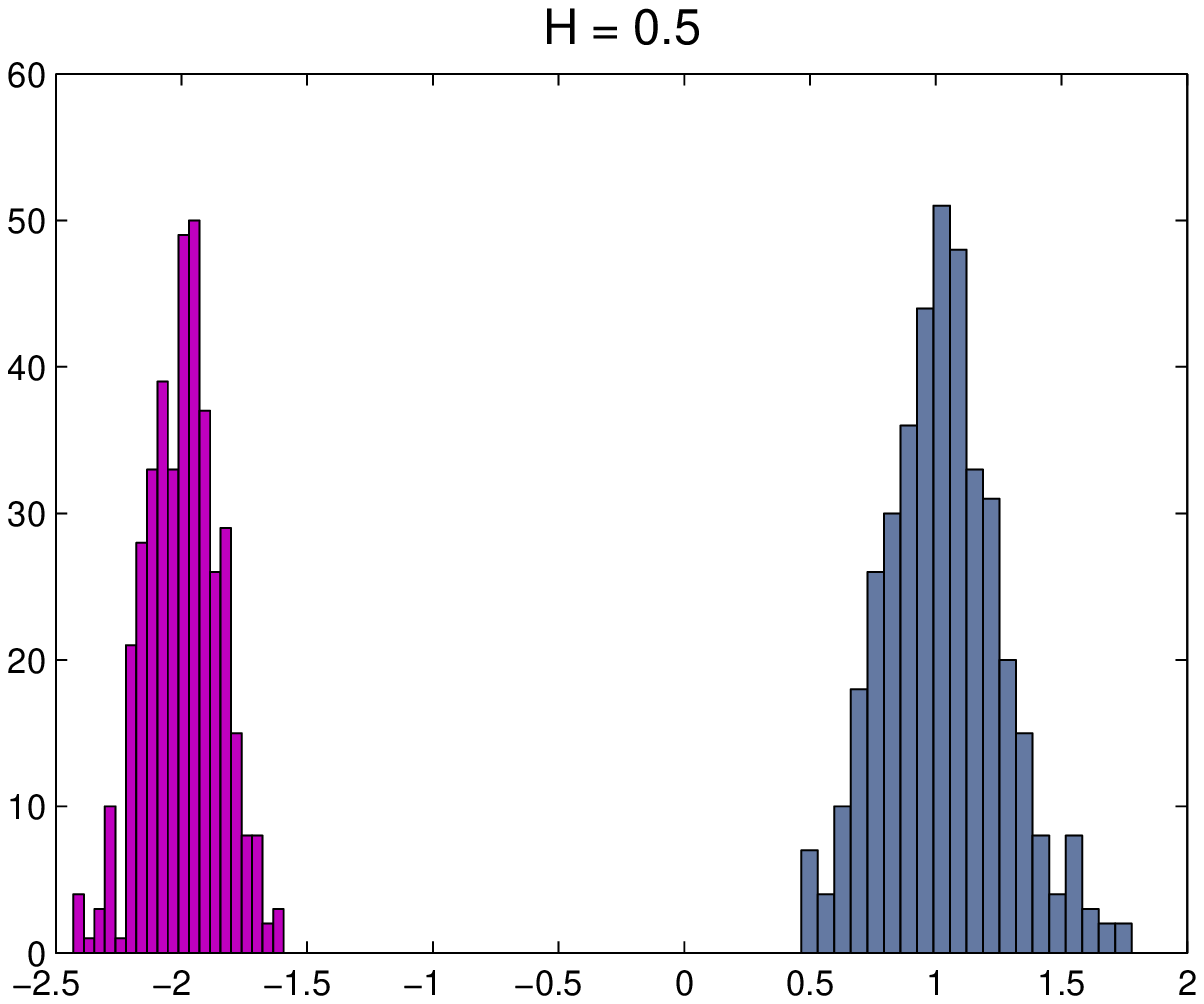}\hspace*{-0.27cm}
 \includegraphics[scale=1,totalheight=4cm,width=0.35\textwidth  ]{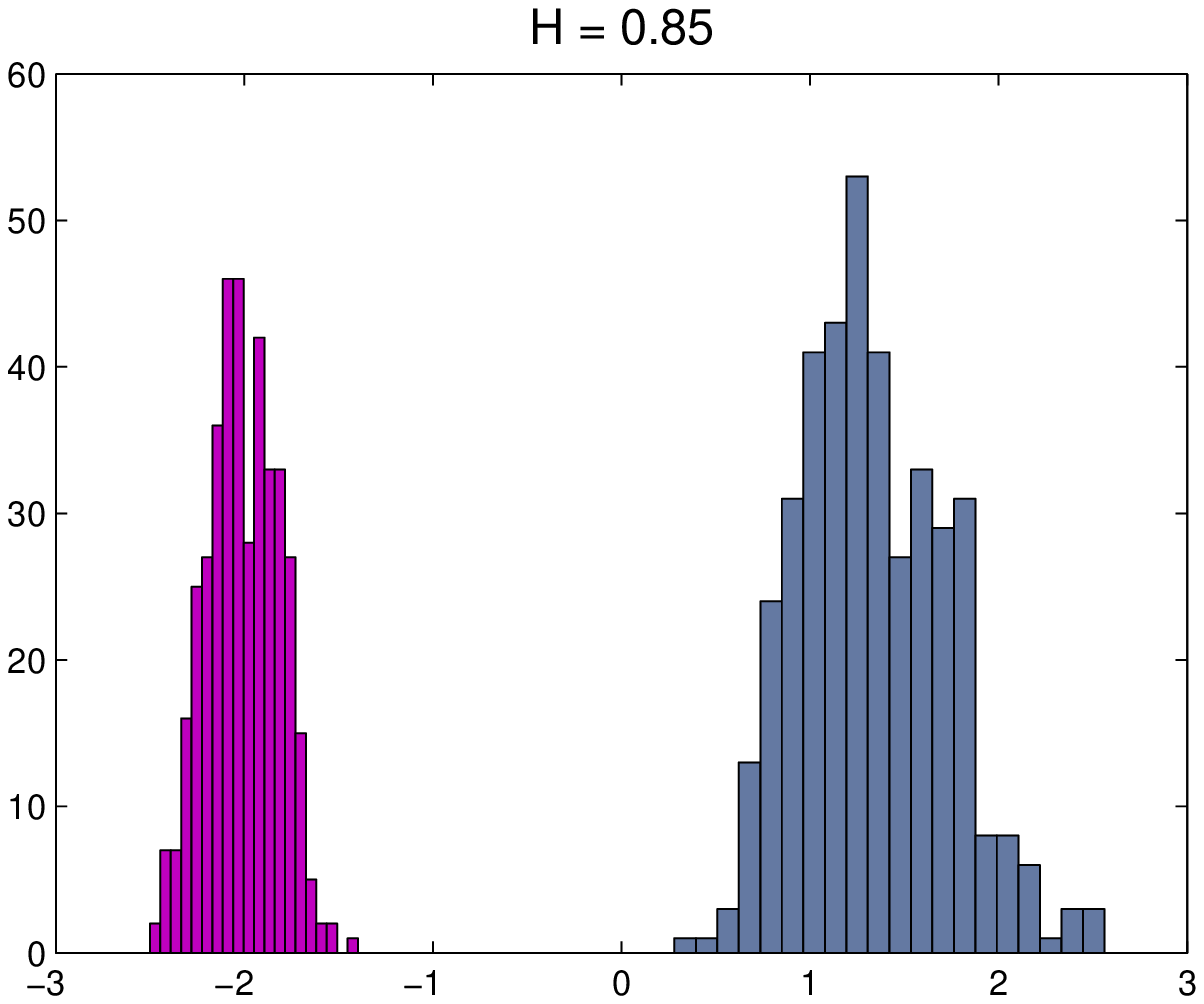}\hspace*{-0.27cm}\\
    \includegraphics[scale=1,totalheight=4cm,width=0.35\textwidth ]{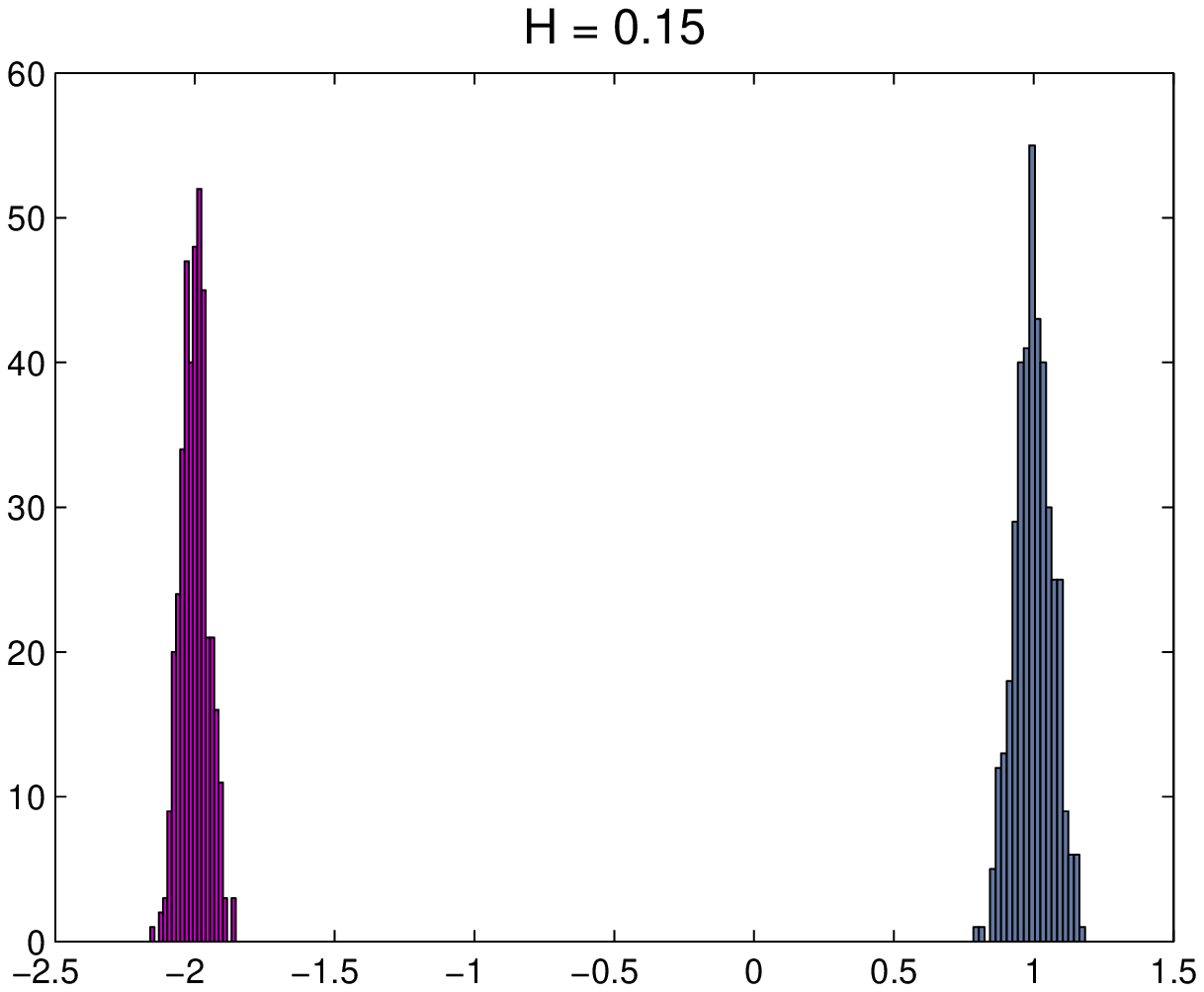}\hspace*{-0.27cm}
  \includegraphics[scale=1,totalheight=4cm,width=0.35\textwidth ]{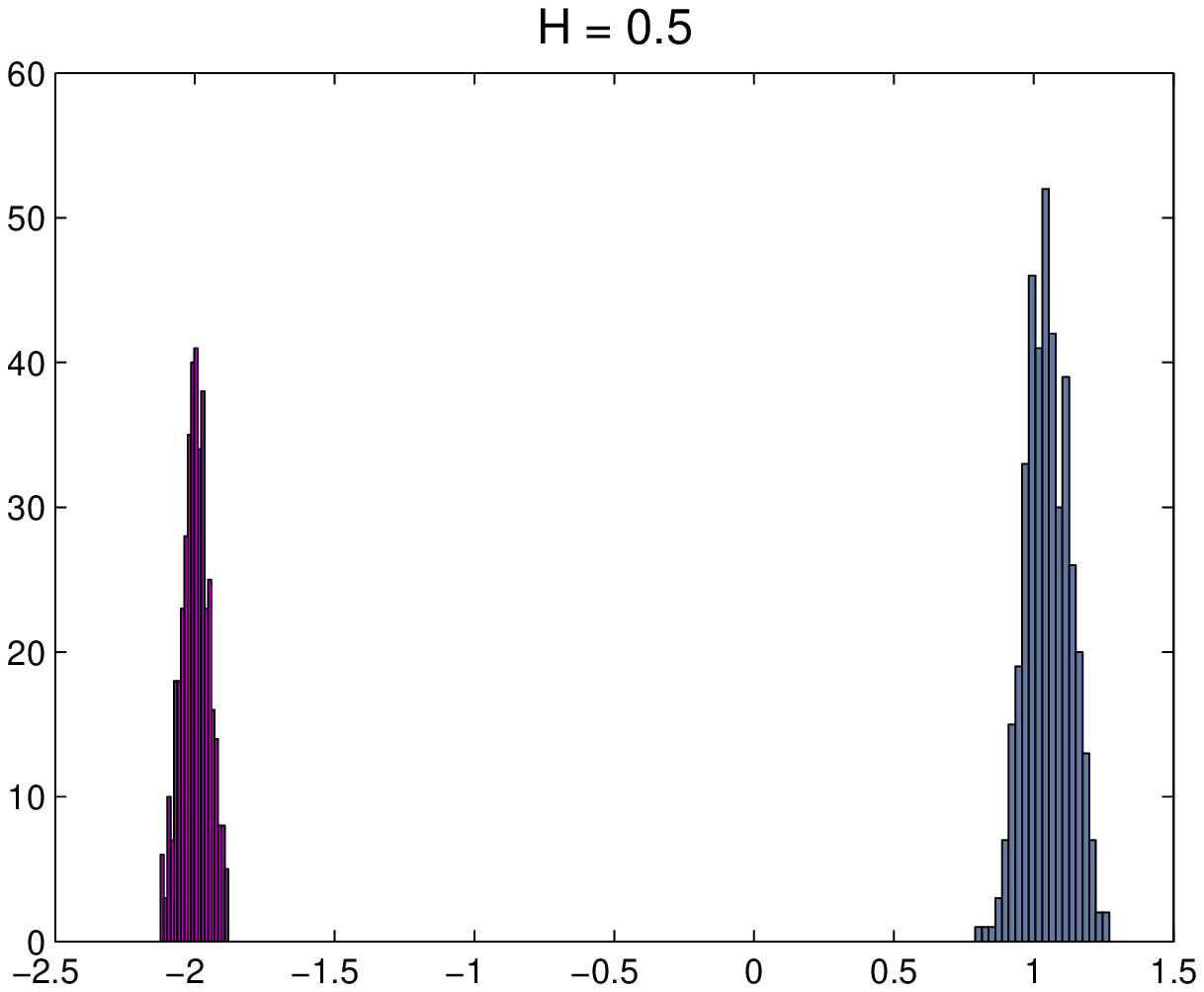}\hspace*{-0.27cm}
   \includegraphics[scale=1,totalheight=4cm ,width=0.35\textwidth ]{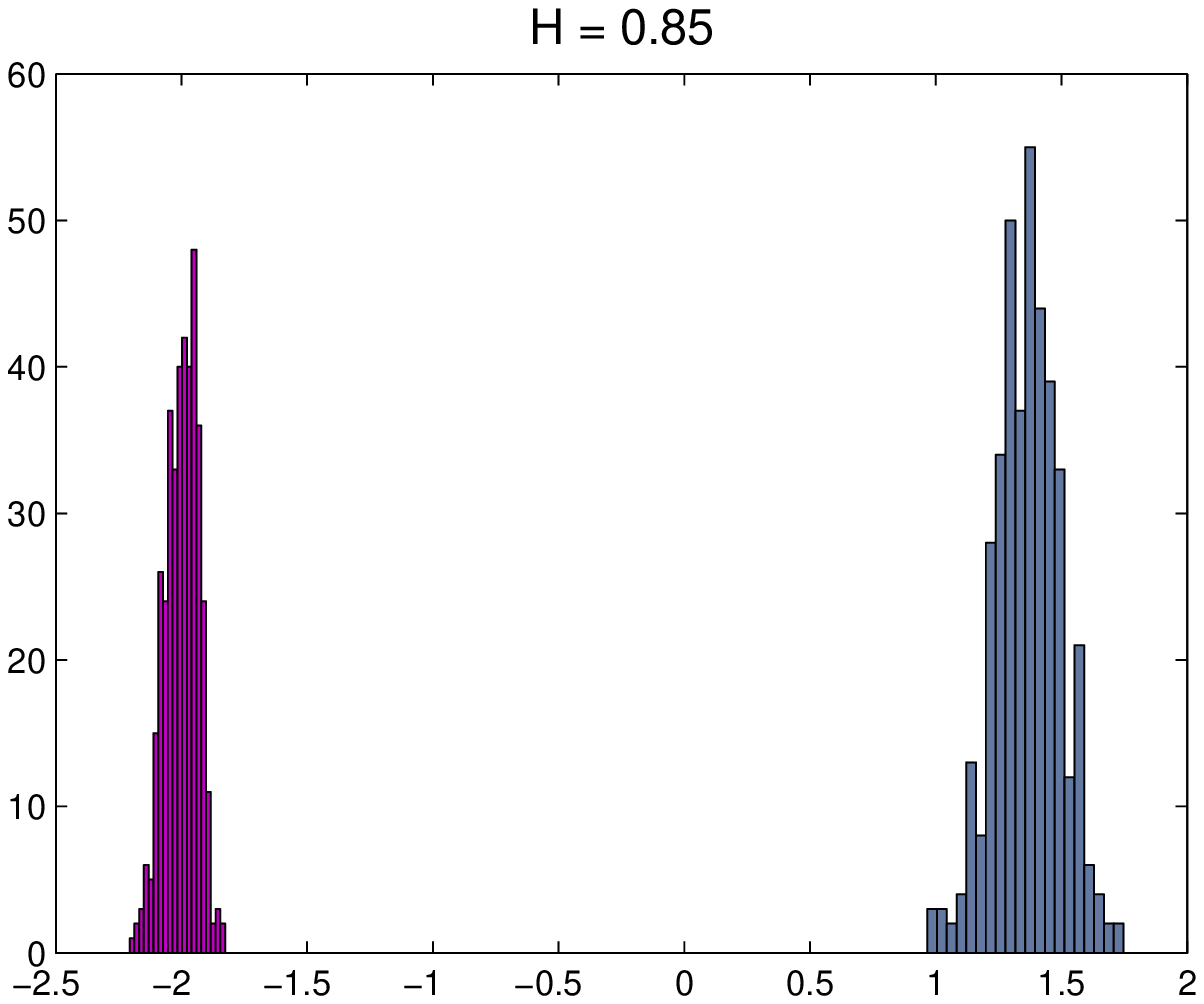}
  \end{center}
  \caption{Frequency histograms of population parameter estimates based on $400$ samples for different values of $(N,H)$. In each box of the two rows  (top $N=50$ and bottom $N=500$) histograms of $\widehat{\mu}$ (pink) and $\widehat{\sigma^2}$ (gray) are given for fixed parameters $(\mu,\sigma^2,T,\bm{n})=(-2,1,5,\bm{2^2})$. (For interpretation of the references to colour in the legend of this figure, the reader is referred to the electronic version of this article)}
  \label{figure 1}
\end{figure}
\begin{table}
\resizebox{\textwidth}{!}{%
\begin{tabular}{lllll}
\hline
True  values                   \hspace{2cm}& $H=0.15$           & $H=0.50$   & $H=0.85$ \\ \cline{1-1}
\begin{tabular}[c]{@{}l@{}} $N=50$  \end{tabular} ~~~&   \begin{tabular}[c]{@{}l@{}}Mean (Std. dev.)\end{tabular} ~~~&  \begin{tabular}[c]{@{}l@{}}Mean (Std. dev.)\end{tabular} ~~~&  \begin{tabular}[c]{@{}l@{}} Mean and  Std. dev.'s \end{tabular} &  ~~~
\\  \cline{2-4}
\hline
\begin{tabular}[c]{@{}l@{}}$\mu=-2$\\ $\sigma^2=1$\end{tabular} ~~~~& \begin{tabular}[c]{@{}l@{}}
-2.0050     (\color{red}{0.1449} \color{blue}{0.1427})\\~0.9713   (\color{red}{0.2077}    \color{blue}{0.2075})\end{tabular} ~~~~& \begin{tabular}[c]{@{}l@{}}~-2.0146       (\color{red}{0.1549}     \color{blue}{0.1518})\\ ~ 1.0028     (\color{red}{0.2376}  \color{blue}{0.2247})   \end{tabular} ~~~~& \begin{tabular}[c]{@{}l@{}} ~-1.9824       (\color{red}{0.1793}    \color{blue}{0.1920})    \\~1.0871       (\color{red}{0.3181} \color{blue}{0.3391})\end{tabular} &
\\
            $N=500$                       &                                   &  &  \\ \cline{1-1}
\begin{tabular}[c]{@{}l@{}}$\mu=-2$\\ $\sigma^2=1$\end{tabular}   ~~~~& \begin{tabular}[c]{@{}l@{}}
-2.0057       (\color{red}{0.0458} \color{blue}{0.0434})\\~    1.0005       (\color{red}{0.0663}    \color{blue}{0.0671})\end{tabular} ~~~~& \begin{tabular}[c]{@{}l@{}}  -1.9979        (\color{red}{0.0490}    \color{blue}{0.0498})\\~ 1.0021        (\color{red}{0.0758}    \color{blue}{0.0758})\end{tabular} ~~~~& \begin{tabular}[c]{@{}l@{}}~-2.0038       (\color{red}{0.0567}  \color{blue}{0.0596}) \\~    1.0849     (\color{red}{0.1015}    \color{blue}{0.1011})\end{tabular} &  ~~~~\\
\hline
\end{tabular}}
\label{Table 2}\caption{The means with exact (red) and empirical (blue) standard deviations of estimators $\widehat{\mu}$, $\widehat{\sigma^2}$ based on $400$ samples, with true values $(\mu_0,\sigma^2_0)=(-2,1)$, $\bm{(T,n)=(5,2^5)}$, and different values of $N$ ($=50;500$).(For interpretation of the references to colour in this table the reader is referred to the electronic version of
this article).}
\end{table}

\begin{figure}{}
  \begin{center}
    \includegraphics[scale=1,totalheight=4cm,width=0.35\textwidth ]{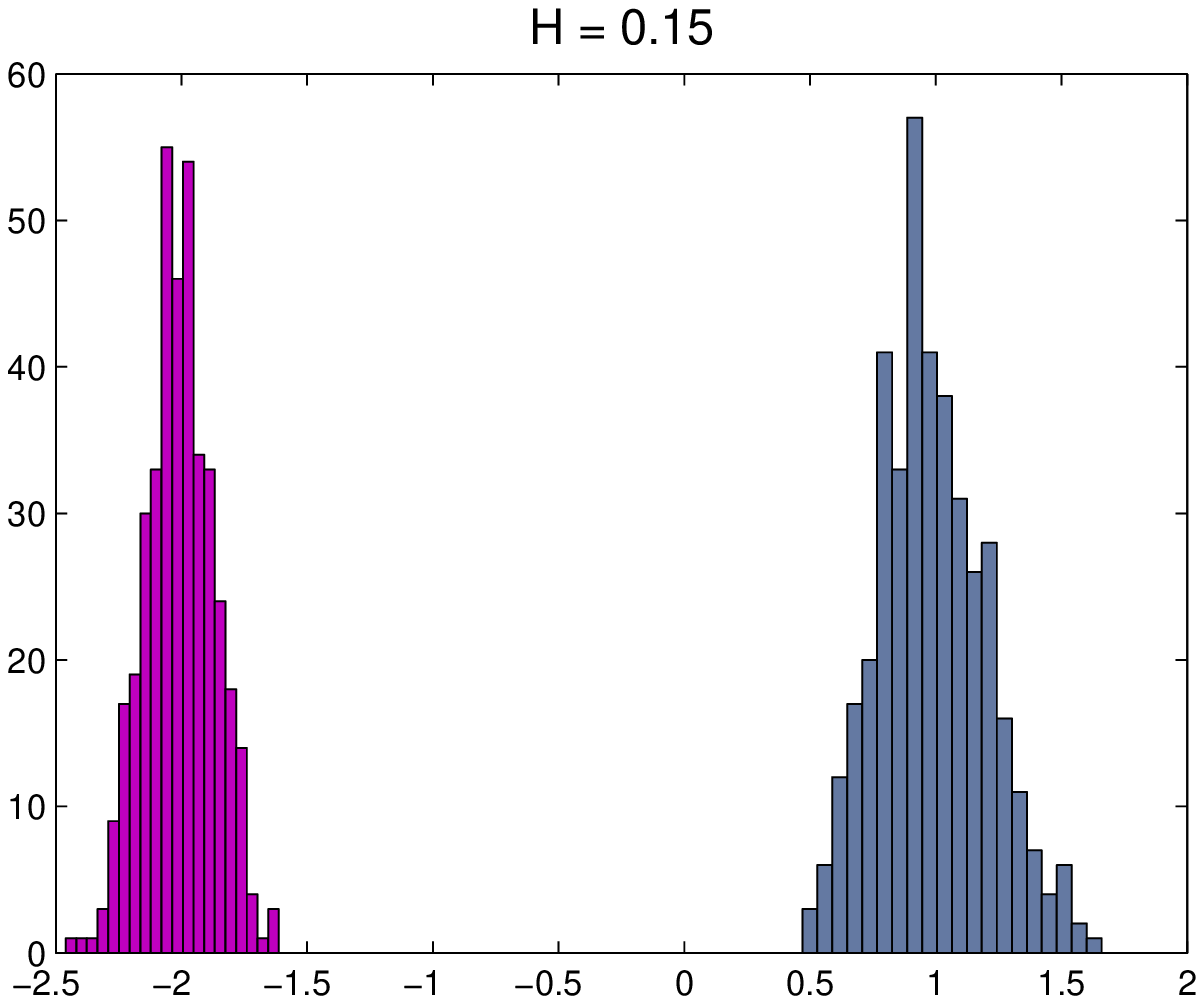}\hspace*{-0.27cm}
   \includegraphics[scale=1,totalheight=4cm,width=0.35\textwidth ]{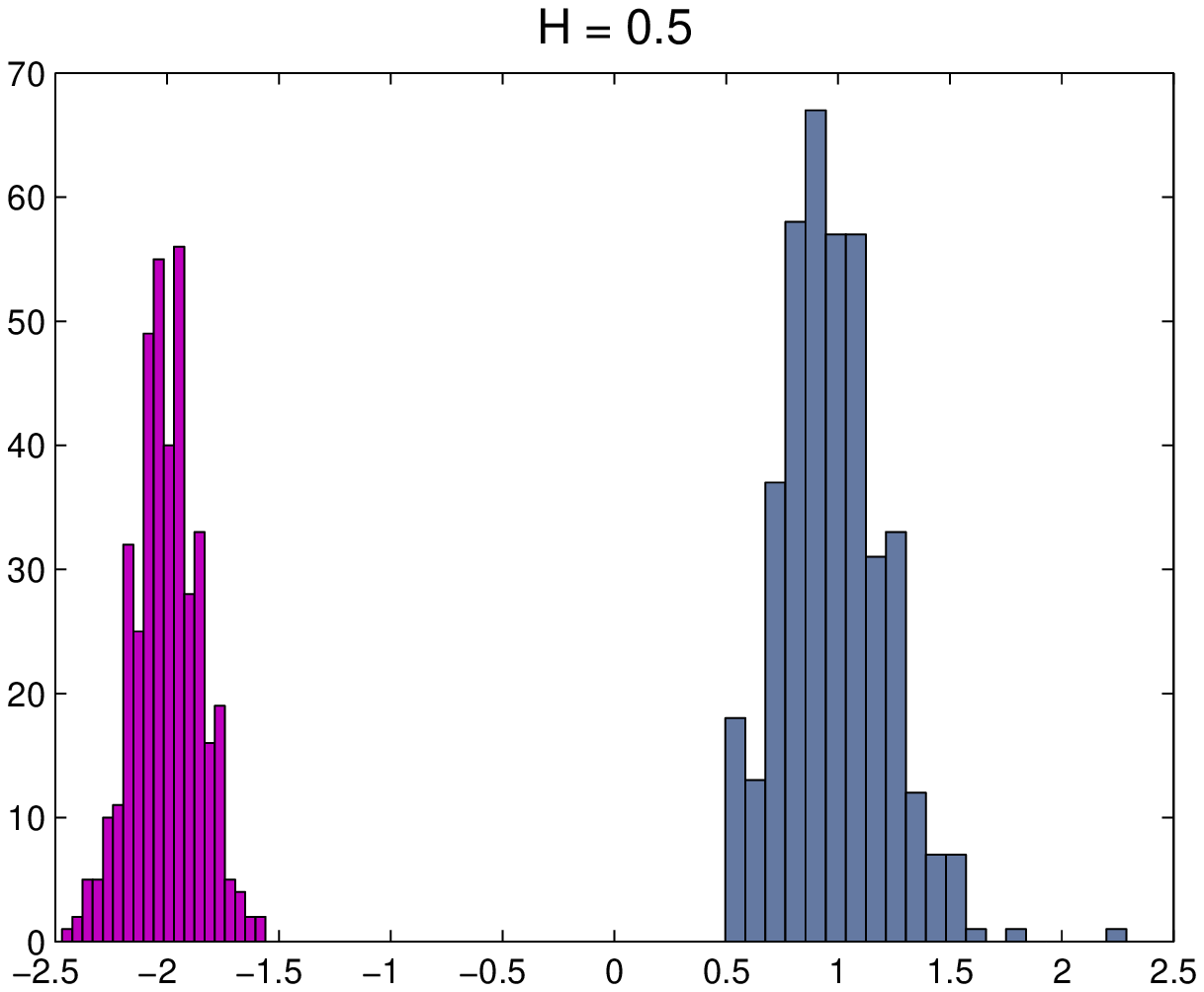}\hspace*{-0.27cm}
   \includegraphics[scale=1,totalheight=4cm ,width=0.35\textwidth ]{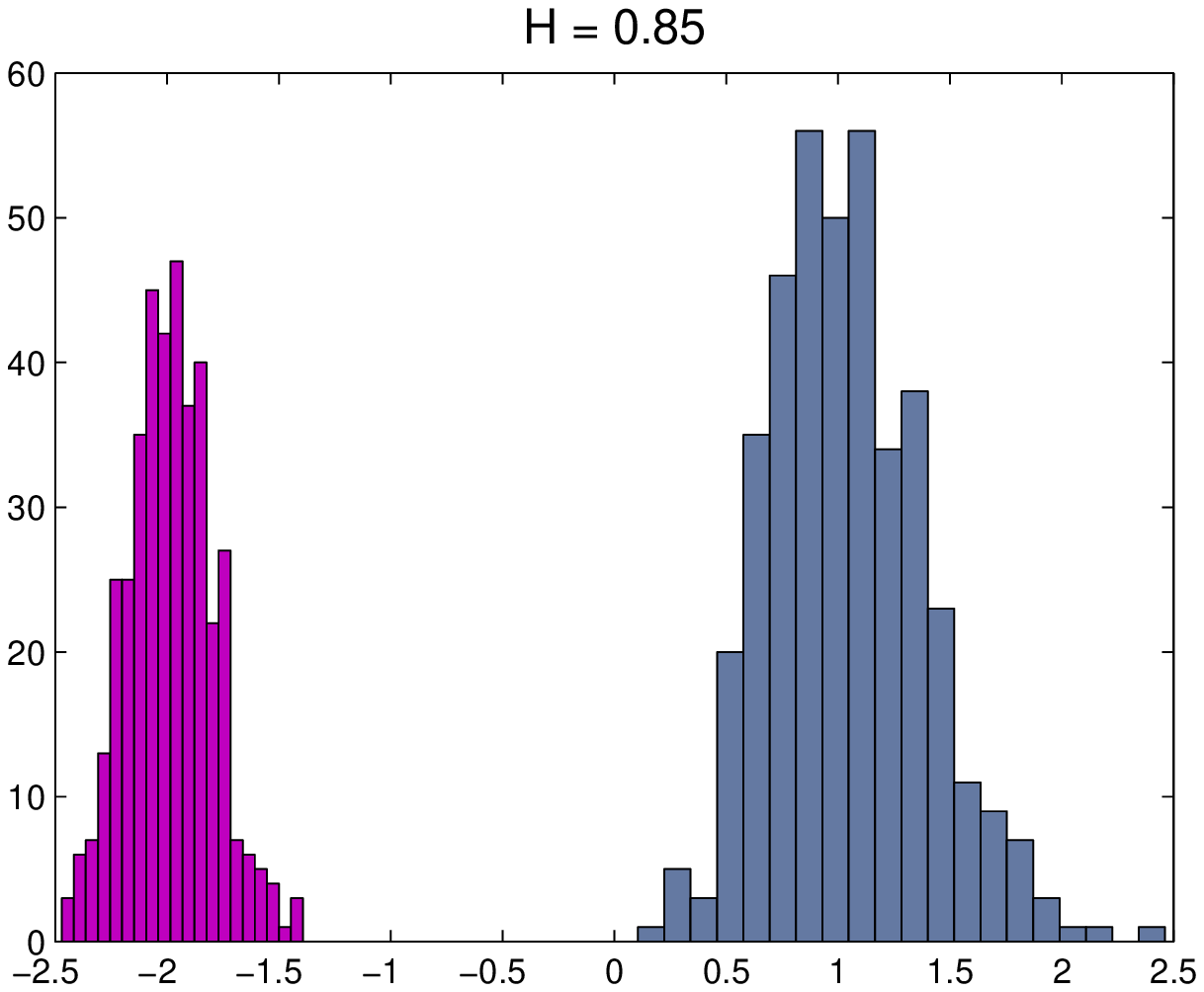}\\
    \includegraphics[scale=1,totalheight=4cm,width=0.35\textwidth ]{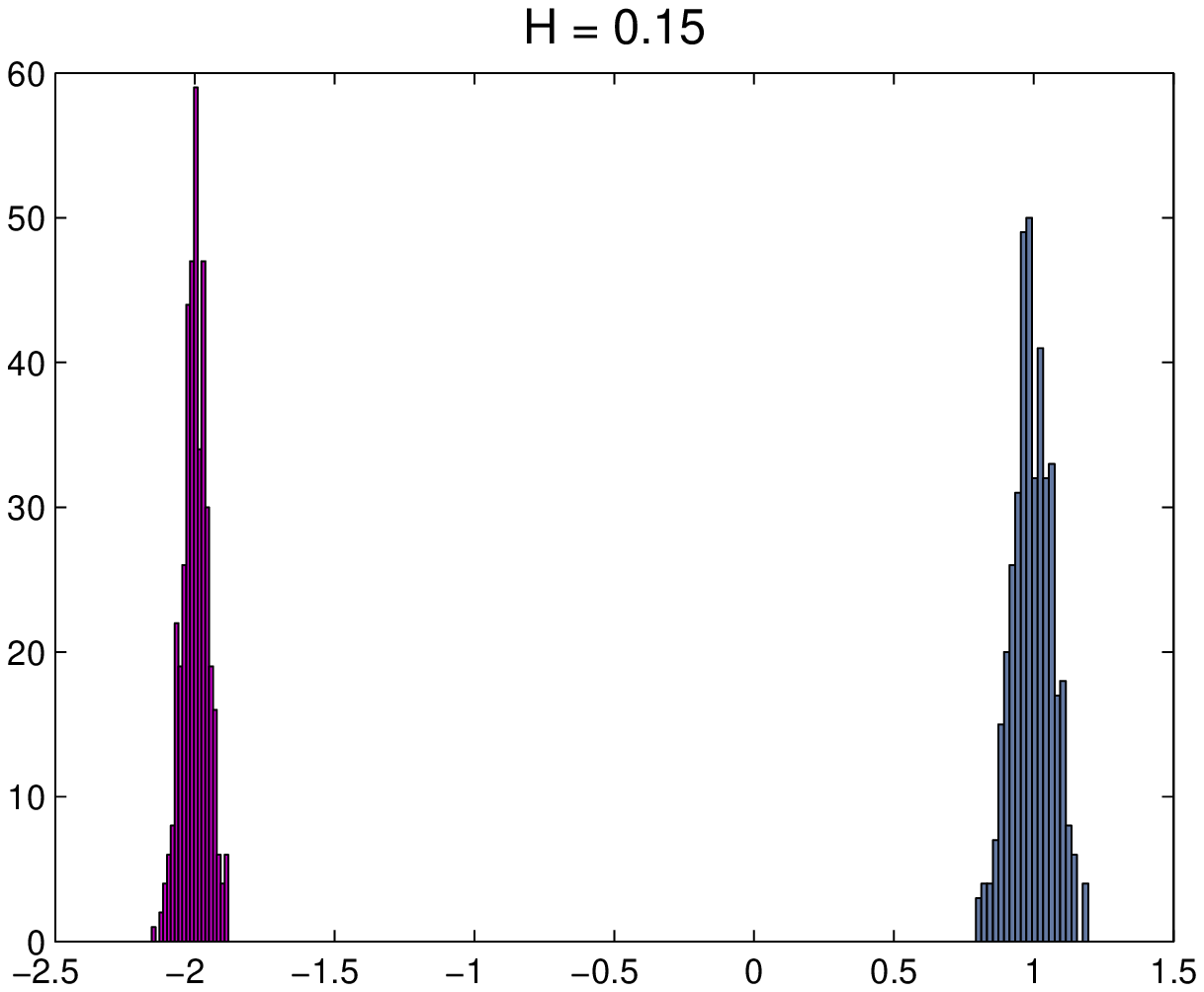}\hspace*{-0.27cm}
   \includegraphics[scale=1,totalheight=4cm,width=0.35\textwidth ]{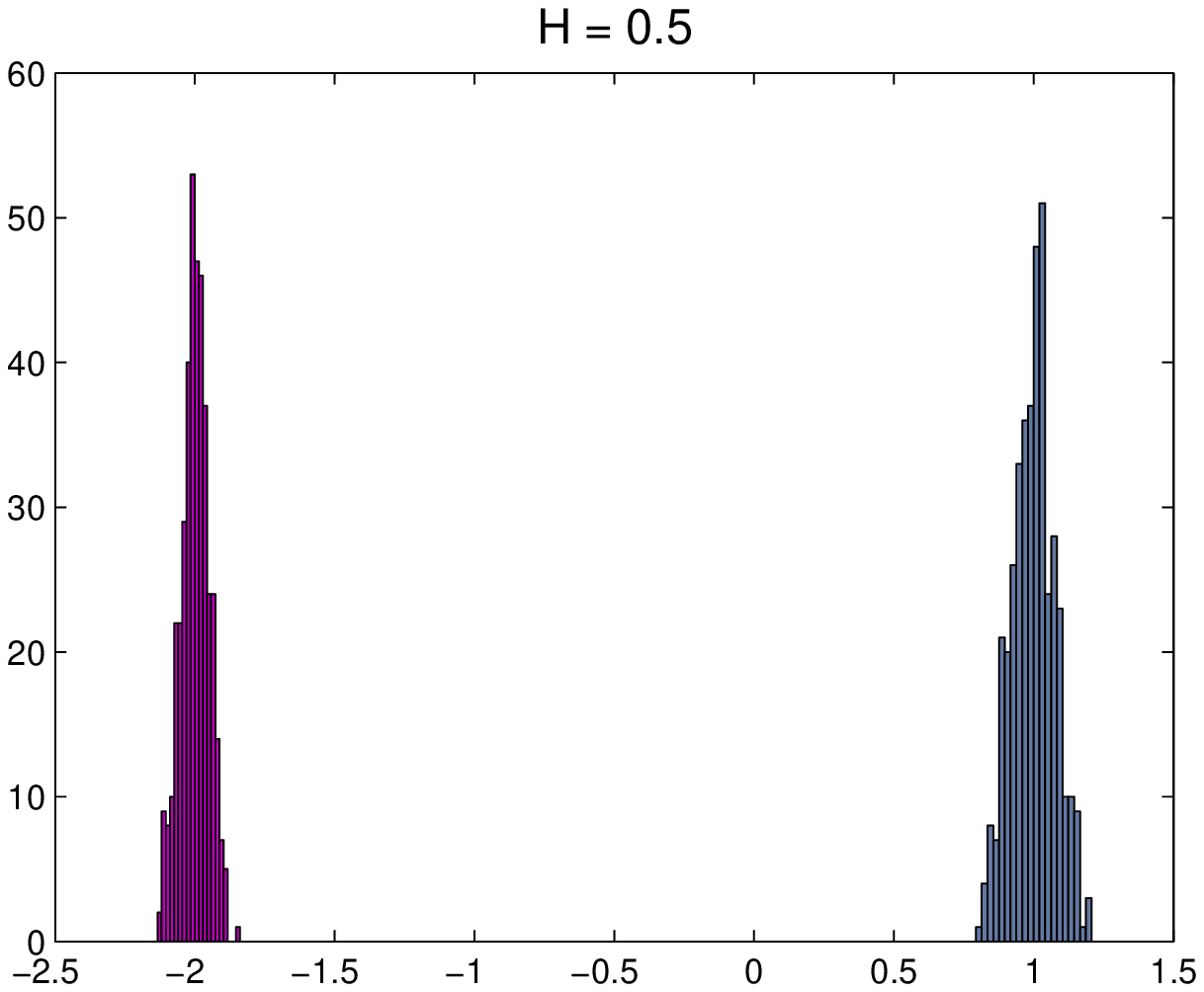}\hspace*{-0.27cm}
   \includegraphics[scale=1,totalheight=4cm ,width=0.35\textwidth ]{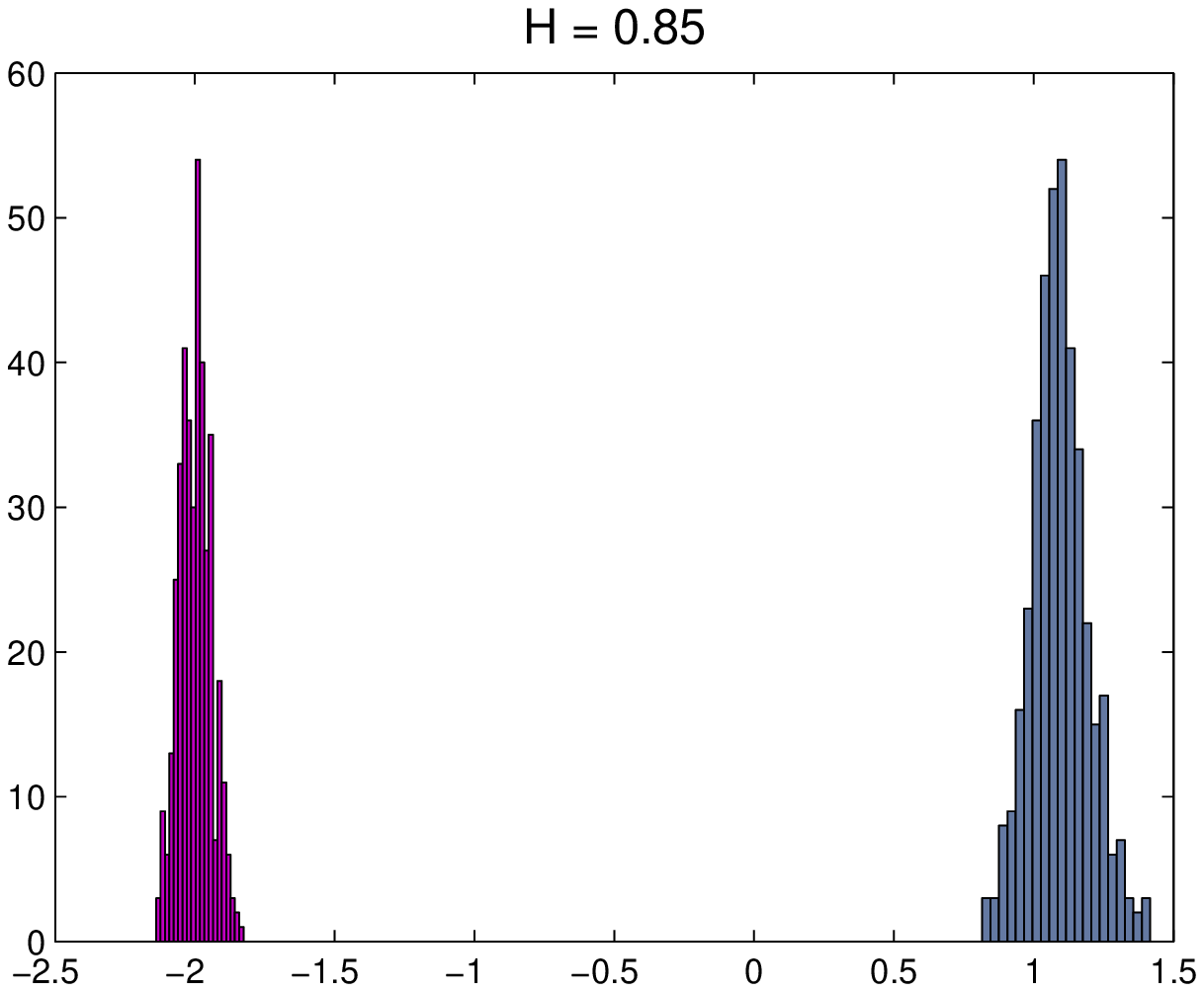}
 \caption{Frequency histograms of population parameter estimates based on $400$ samples for different values of $(N,H)$. In each box of the two rows  (top $N=50$ and bottom $N=500$) histograms of $\widehat{\mu}$ (pink) and $\widehat{\sigma^2}$ (gray) are given for fixed parameters $(\mu,\sigma^2,T,\bm{n})=(-2,1,5,\bm{2^5})$. (For interpretation of the references to colour in the legend of this figure, the reader is referred to the electronic version of this article). }
   \end{center}
 \label{figure 2}
\end{figure}

\begin{table}
\resizebox{\textwidth}{!}{%
\begin{tabular}{lllll}
\hline
True  values                   \hspace{2cm}& $H=0.15$           & $H=0.50$   & $H=0.85$ \\ \cline{1-1}
\begin{tabular}[c]{@{}l@{}} $N=50$  \end{tabular} ~~~&   \begin{tabular}[c]{@{}l@{}}Mean (Std. dev.)\end{tabular} ~~~&  \begin{tabular}[c]{@{}l@{}}Mean (Std. dev.)\end{tabular} ~~~&  \begin{tabular}[c]{@{}l@{}} Mean and  Std. dev.'s \end{tabular} &  ~~~
\\  \cline{2-4}
\hline
\begin{tabular}[c]{@{}l@{}}$\mu=-2$\\ $\sigma^2=1$\end{tabular} ~~~~& \begin{tabular}[c]{@{}l@{}}
-2.0015     (\color{red}{0.1447} \color{blue}{0.1454})\\~0.9996   (\color{red}{0.2073}    \color{blue}{0.2008})\end{tabular} ~~~~& \begin{tabular}[c]{@{}l@{}}~-1.9960       (\color{red}{0.1549}     \color{blue}{0.1563})\\ ~ 0.9764     (\color{red}{0.2376}  \color{blue}{0.2448})   \end{tabular} ~~~~& \begin{tabular}[c]{@{}l@{}} ~-2.0055   (\color{red}{0.1792}    \color{blue}{0.1709})    \\~0.9971       (\color{red}{0.3180} \color{blue}{0.3323})\end{tabular} &
\\
            $N=500$                       &                                   &  &  \\ \cline{1-1}
\begin{tabular}[c]{@{}l@{}}$\mu=-2$\\ $\sigma^2=1$\end{tabular}   ~~~~& \begin{tabular}[c]{@{}l@{}}
-1.9997       (\color{red}{0.0458} \color{blue}{0.0442})\\~    0.9971       (\color{red}{0.0662}    \color{blue}{0.0650})\end{tabular} ~~~~& \begin{tabular}[c]{@{}l@{}}  -2.0009        (\color{red}{0.0490}    \color{blue}{0.0471})\\~ 0.9993        (\color{red}{0.0758}    \color{blue}{0.0747})\end{tabular} ~~~~& \begin{tabular}[c]{@{}l@{}}~-2.0006       (\color{red}{0.0567}  \color{blue}{0.0566}) \\~    1.0083       (\color{red}{0.1015}    \color{blue}{0.1045})\end{tabular} &  ~~~~\\
\hline
\end{tabular}}
\caption{The means with exact (red) and empirical (blue) standard deviations of estimators $\widehat{\mu}$, $\widehat{\sigma^2}$ based on $400$ samples, with true values $(\mu_0,\sigma^2_0)=(-2,1)$, $\bm{(T,n)=(5,2^8)}$, and different values of $N$ ($=50;500$).(For interpretation of the references to colour in this table the reader is referred to the electronic version of
this article).} \label{Table 3}
\end{table}
\begin{figure}{}
  \begin{center}
    \includegraphics[scale=1,totalheight=4cm,width=0.35\textwidth ]{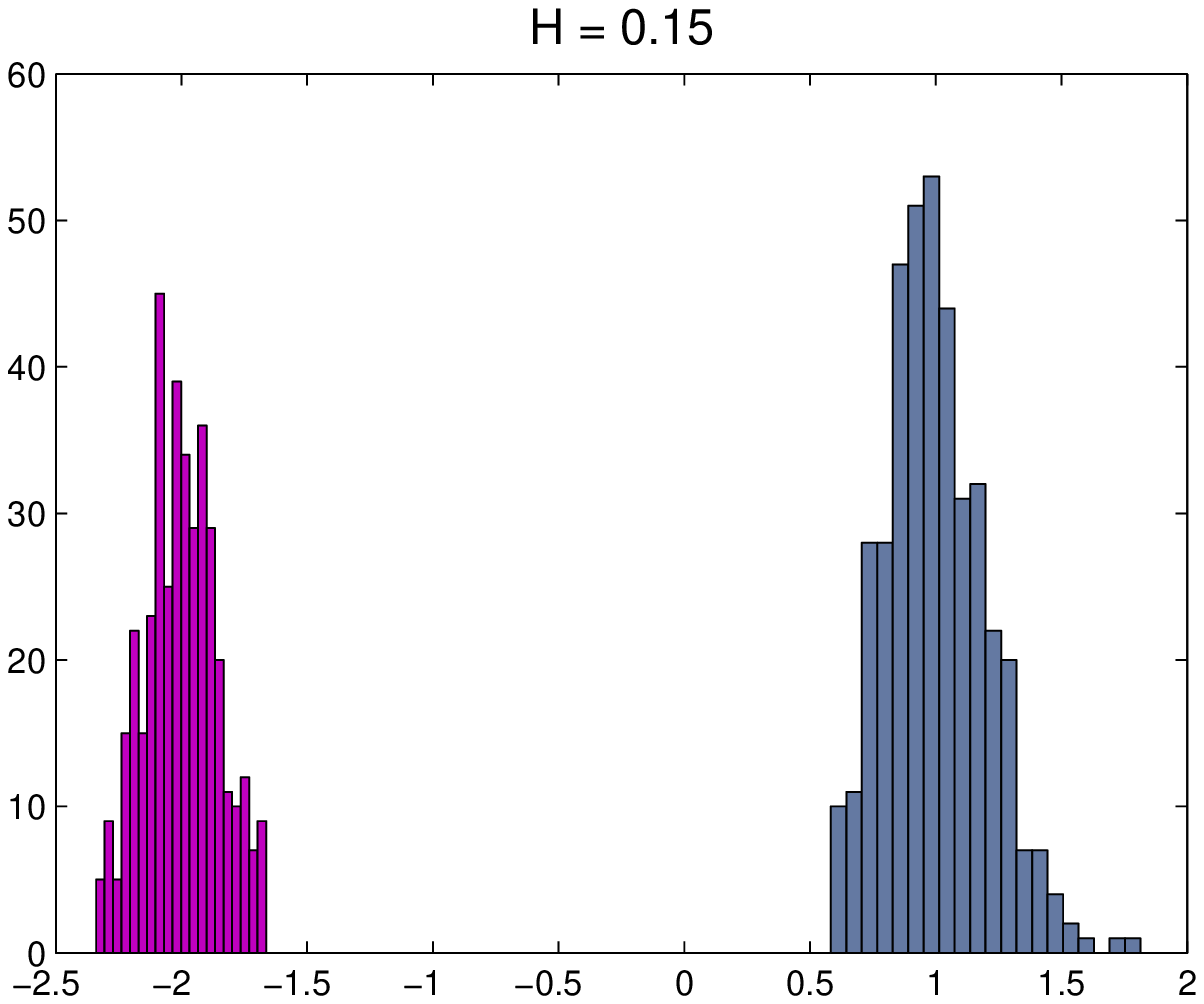}\hspace*{-0.27cm}
  \includegraphics[scale=1,totalheight=4cm,width=0.35\textwidth ]{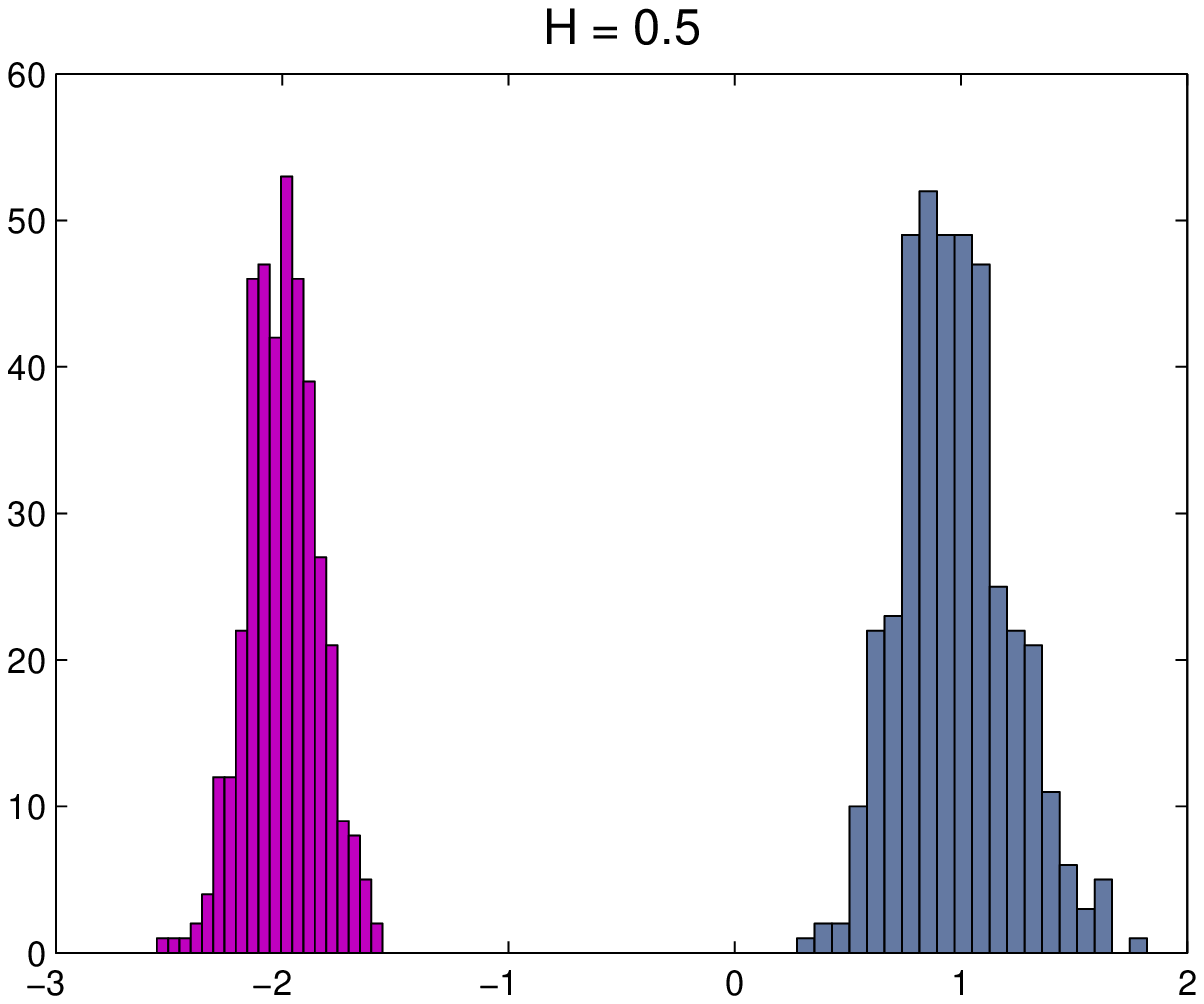}\hspace*{-0.27cm}
   \includegraphics[scale=1,totalheight=4cm ,width=0.35\textwidth ]{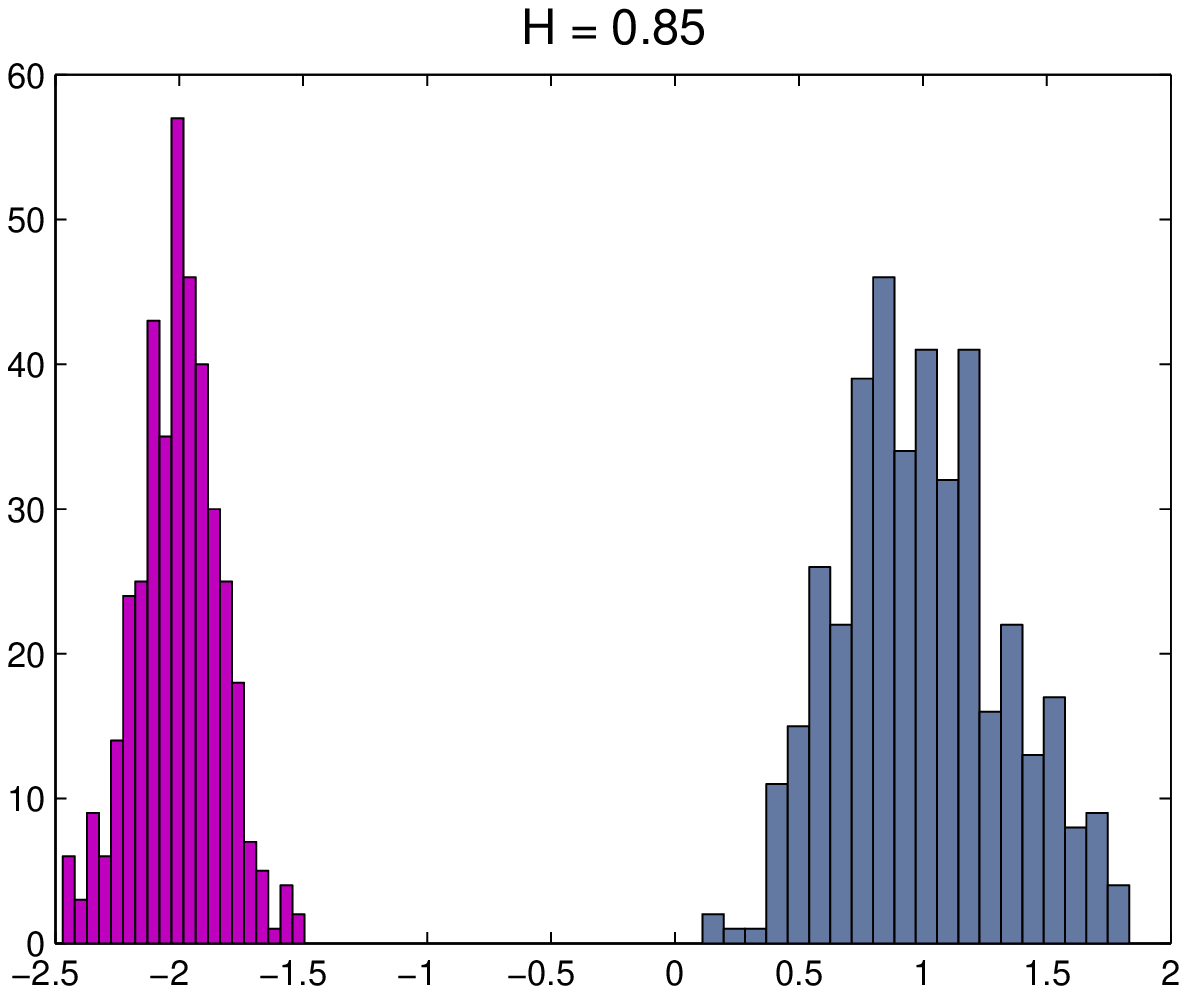}
\\
   \includegraphics[scale=1,totalheight=4cm,width=0.35\textwidth ]{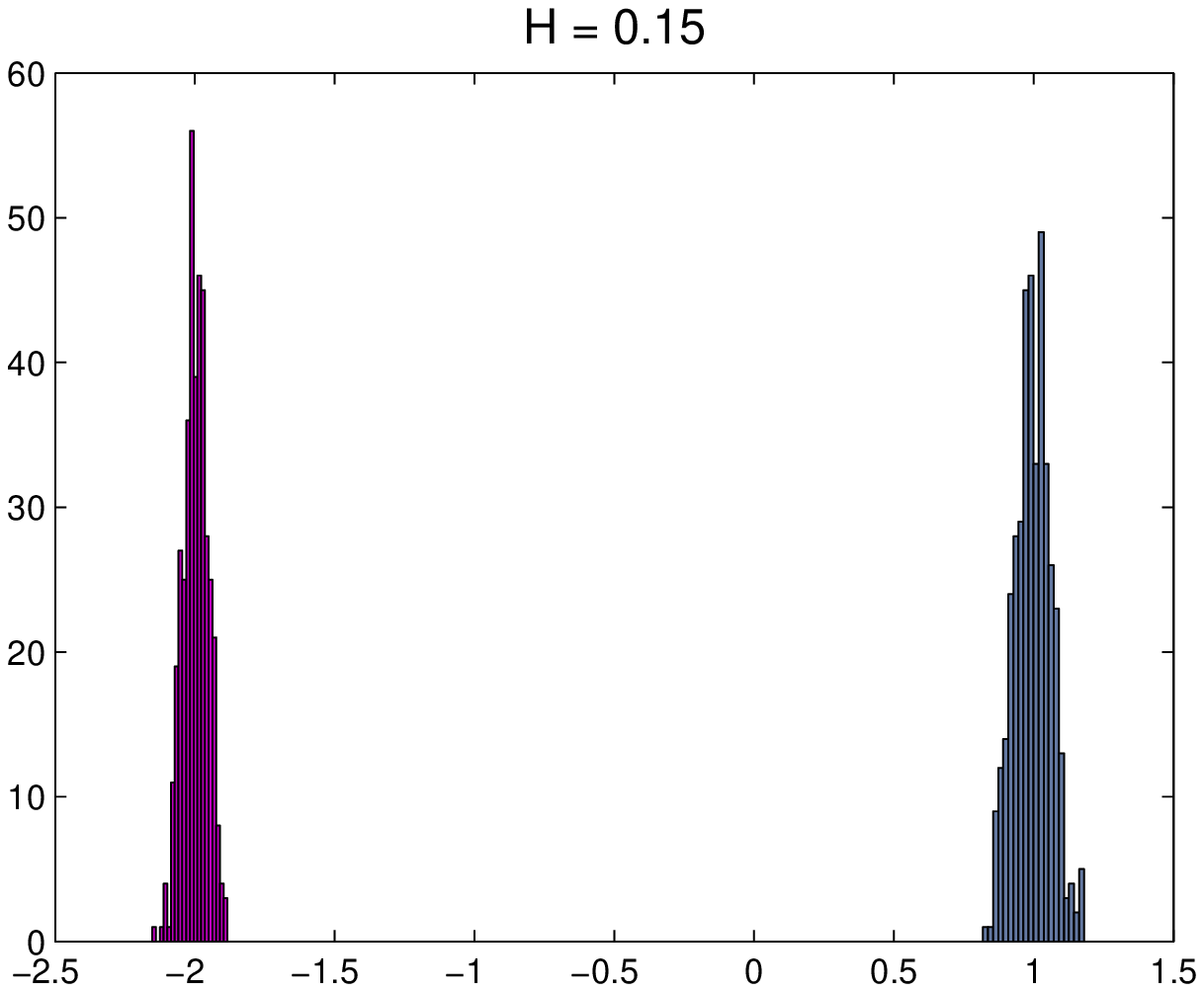}\hspace*{-0.27cm}
  \includegraphics[scale=1,totalheight=4cm,width=0.35\textwidth ]{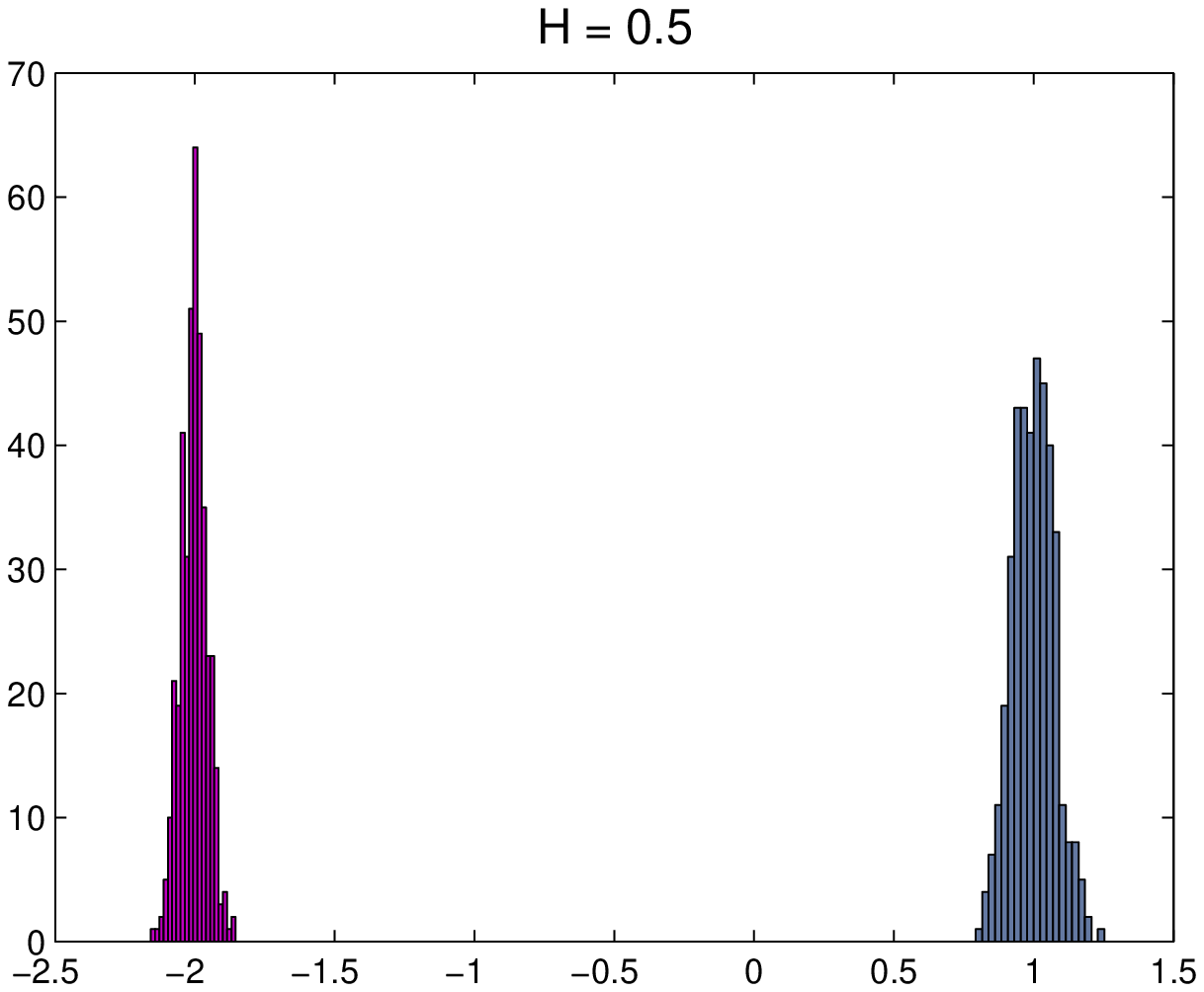}\hspace*{-0.27cm}
  \includegraphics[scale=1,totalheight=4cm ,width=0.35\textwidth ]{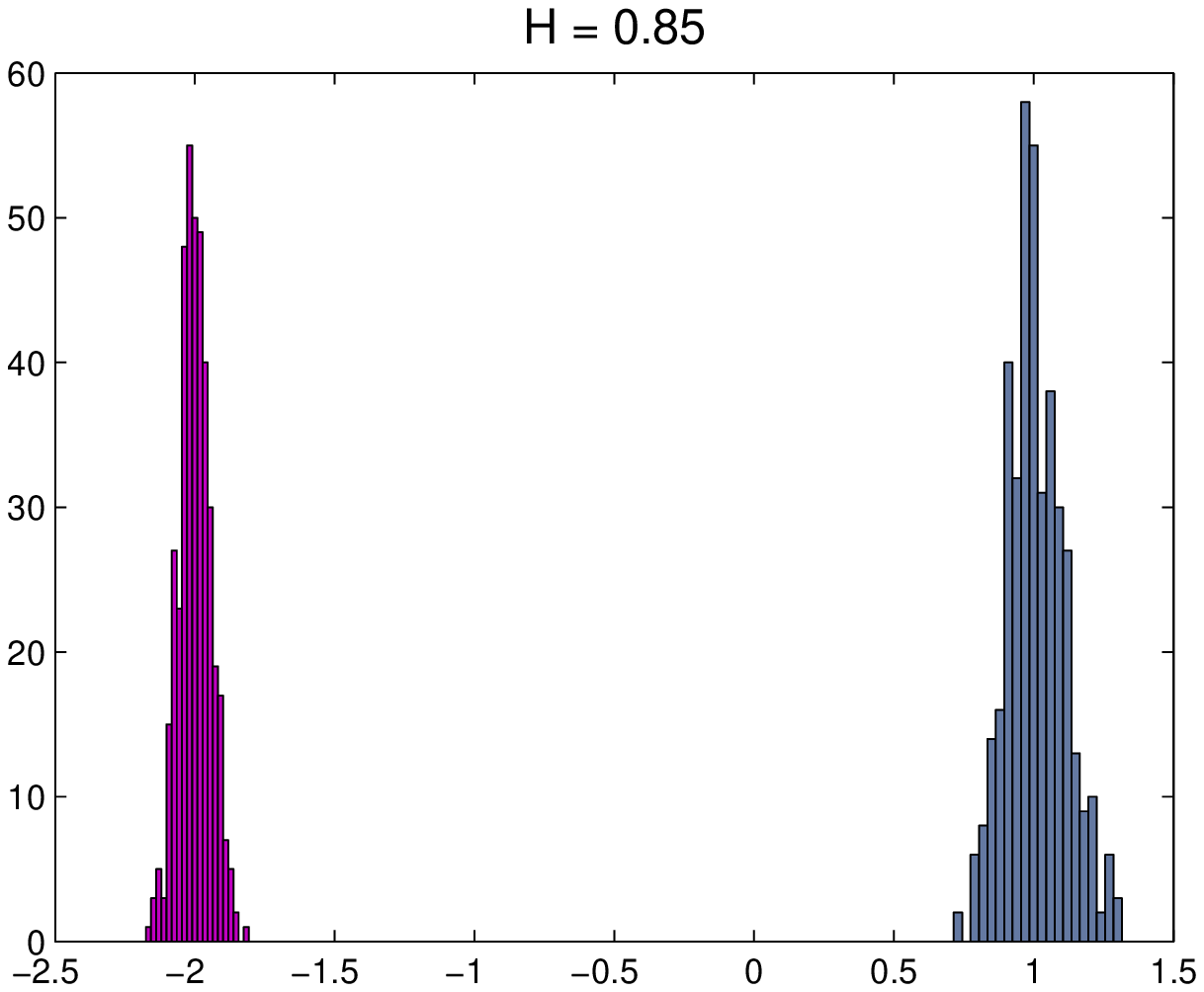}
 \end{center}
 \caption{Frequency histograms of population parameter estimates based on $400$ samples for different values of $(N,H)$. In each box of the two rows  (top $N=50$ and bottom $N=500$) histograms of $\widehat{\mu}$ (pink) and $\widehat{\sigma^2}$ (gray) are given for fixed parameters $(\mu,\sigma^2,T,\bm{n})=(-2,1,5,\bm{2^8})$. (For interpretation of the references to colour in the legend of this figure, the reader is referred to the electronic version of this article).}
 \label{figure 3}
\end{figure}
\section{Concluding remarks}\label{Sec.5}
In this paper, we have provided a fully Likelihood  parametric estimation of  population parameter for specific dynamical  model described by a fractional SDE including a random effects in the drift. We are essentially concerned with the estimation of  Hurst index, as well as  with the  mean and variance estimators of the  random effect  that has a Gaussian distribution. All qualitative and  asymptotic property of the estimators are obtained when the population of subjects  becomes large.\bigskip

This study  suggests several important directions for future research. First, what are the asymptotic properties of the Maximum Likelihood estimators for $\mu$ and $\sigma^2$  when the Hurst index $H$ is unknown? Given that the model is fully parameterized,  one may wish to estimate $H$, $\mu$ and $\sigma^2$ simultaneously. The  achievement of this task is an ongoing work. Second, the present study assumes that the model is linear and the diffusion is constant and equals  1. This assumption is not verified  in almost all real applications. So to overcome this issue, one can use for example, Euler schemes approximation. However, it is not clear how to get an explicit approximation for the Maximum Likelihood function. Such  extension would be worth being studied from both theoretical and application points of view. Third, as mentioned previously, we may estimate the population parameters by using double asymptotic framework. Such an idea is considered in an ongoing work for a more general model and in the nonparametric estimation context.\bigskip

\paragraph*{acknowledgements}
The anonymous referees and the editors are acknowledged for their constructive comments and suggestions, which have greatly improved the paper. The three  authors acknowledge the financial support of the Moroccan-German Scientific and Technics Programme; Project: {\bf PMARS III} 060/2015.
%
%

\end{document}